\documentclass{amsart}

\theoremstyle{plain}
    \newtheorem{theorem}{Theorem}[section]
    \newtheorem{proposition}[theorem]{Proposition}
    \newtheorem{lemma}[theorem]{Lemma}
    \newtheorem{corollary}[theorem]{Corollary}
    \newtheorem{conjecture}[theorem]{Conjecture}
\theoremstyle{definition}
    \newtheorem{definition}[theorem]{Definition}
    
    \newtheorem{remark}[theorem]{Remark}

\usepackage{amsmath}
\usepackage{amssymb}
\usepackage{mathrsfs} 
\usepackage{times}
\usepackage{amscd}
\usepackage[all]{xy}

\begin{document}

\title[Regulator of Fermat motives]{On the regulator of Fermat motives and generalized hypergeometric functions}
\author{Noriyuki Otsubo} 
\address{Department of Mathematics and Informatics, Chiba University, 
Yayoicho 1-33, Inage, Chiba, 263-8522 Japan. }
\email{otsubo@math.s.chiba-u.ac.jp}
\date{September 3, 2009}
\begin{abstract}
We calculate the Beilinson regulators of motives associated to Fermat curves  
and express them by special values of generalized hypergeometric functions.  
As a result, we obtain surjectivity results of the regulator,     
which support the Beilinson conjecture on special values of $L$-functions.  
\end{abstract}

\maketitle

\numberwithin{equation}{section}
\renewcommand{\labelenumi}{\rm{(\roman{enumi})}}

\def\endpiece{xxx}
\def\mkop#1{\expandafter\def\csname #1\endcsname{\operatorname{#1}}}
\def\xmakeop#1,{\def\temp{#1}\ifx\temp\endpiece\else\mkop{#1}\expandafter\xmakeop\fi}
\def\makeop[#1]{\xmakeop#1,xxx,}
\makeop[Aut,alt,arg,End,Ext,Hom,sym,Mor,Ker,Im,Coker,Cone,id,Sm,Var,Gal,Spec,tor,div,CH,Pic,DM,alb,Corr,Chow,Ind,ord,rank,sgn,Tot,tr,Tr,Re,Res,Jac]
\def\mkmathrm#1{\expandafter\def\csname #1\endcsname{\mathrm{#1}}}
\def\xmakemathrm#1,{\def\temp{#1}\ifx\temp\endpiece\else\mkmathrm{#1}\expandafter\xmakemathrm\fi}
\def\makemathrm[#1]{\xmakemathrm#1,xxx,}
\makemathrm[Betti,dR,Hodge,pr,prim]
\def\et{\mathrm{\acute{e}t}}

\def\mkBbb#1{\expandafter\def\csname #1\endcsname{\mathbb{#1}}}
\def\xmakeBbb#1,{\def\temp{#1}\ifx\temp\endpiece\else\mkBbb{#1}\expandafter\xmakeBbb\fi}
\def\makeBbb[#1]{\xmakeBbb#1,xxx,}
\makeBbb[A,C,F,H,P,Q,R,Z]
\def\Gm{{\Bbb G}_{m}}

\def\a{\alpha}\def\b{\beta}\def\g{\gamma}\def\d{\delta}\def\pd{\partial}\def\io{\iota}
\def\l{\lambda}\def\om{\omega}\def\s{\sigma}
\def\t{\tau}\def\th{\theta}\def\x{\xi}\def\z{\zeta}\def\vphi{\varphi}
\def\Om{\Omega}\def\vG{\varGamma}\def\vL{\varLambda}\def\vO{\varOmega}
\def\nb{\nabla}

\def\ol#1{\overline{#1}}\def\ul#1{\underline{#1}}\def\wt#1{\widetilde{#1}}\def\wh#1{\widehat{#1}}
\def\os#1#2{\overset{#1}{#2}}\def\us#1#2{\underset{#1}{#2}}
\def\ot{\otimes}\def\op{\oplus}
\def\da{\downarrow}\def\ua{\uparrow}\def\ra{\rightarrow}\def\lra{\longrightarrow}
\def\isom{\os{\simeq}{\ra}}\def\lisom{\os{\simeq}{\lra}}

\def\sB{\mathscr{B}}\def\sD{\mathscr{D}}\def\sH{\mathscr{H}}
\def\sL{\mathscr{L}}\def\sM{\mathscr{M}}
\def\sO{\mathscr{O}}\def\sP{\mathcal{P}}
\def\sX{\mathscr{X}}\def\sY{\mathscr{Y}}\def\sV{\mathscr{V}}

\def\Mod{\operatorname{\mathsf{Mod}}}
\def\Fr{\mathrm{Fr}}
\def\El{{E_\ell}}
\def\angle#1{\langle #1 \rangle}
\def\Fp{{\F_p}}\def\Qp{{\Q_p}}
\def\Zl{{\Z_\ell}}\def\Ql{{\Q_\ell}}

\section{Introduction}

Fermat varieties have been touchstones for various conjectures in number theory and algebraic geometry. 
For example, they have provided evidences for the conjectures of Weil, Hasse-Weil, Hodge, Tate, Bloch, etc.  
Another most fascinating conjecture is the Beilinson conjecture \cite{beilinson} on special values of the $L$-functions of 
motives over number fields. 
It integrates former conjectures of Tate, Birch-Swinnerton-Dyer, Bloch and Deligne, 
and gives us a beautiful perspective on the mysterious but strong connections between the analysis 
($L$-function) and the geometry (cohomology) of motives. 
Unfortunately, we have only limited results on the conjecture and no general approach seems to be known. 
The aim of this paper is to study the conjecture for motives associated to Fermat curves. 

To an algebraic variety, or more generally to a motive $M$ over a number field, 
its $L$-function $L(M,s)$ is defined by a Dirichlet series convergent on a complex right half plane. 
Conjecturally in general, it has an analytic continuation to the whole complex plane and 
satisfies a functional equation with respect to $s \leftrightarrow 1-s$. 
The Beilinson conjecture explains the behavior of the $L$-function at an integer 
 in terms of the regulator map, that is, a canonical map
$$r_\sD \colon H_\sM^\bullet(M,\Q(r))_\Z \lra H_\sD^\bullet(M_\R,\R(r))$$
from the integral part of the motivic cohomology to the real Deligne cohomology 
(see \S 4 for definitions). 
The conjecture asserts, under an assumption on $r$, firstly that $r_\sD \ot_\Q \R$ is an isomorphism,  
which implies by the functional equation that  
$$\dim_\Q H_\sM^\bullet(M,\Q(r))_\Z = \ord_{s=1-r} L(M^\vee,s),$$
where $M^\vee$ is the dual motive. 
For an integer $n$, let $L^*(M,n)$ denote the first non-vanishing Taylor coefficient of $L(M,s)$ at $s=n$. 
Then the second assertion is that  
$$\det(r_\sD) = L^*(M^\vee,1-r)$$
in $\R^*/\Q^*$, 
where the determinant is taken with respect to a canonical $\Q$-structure on the Deligne cohomology.

For example, if $M=M^\vee=\Spec k$, the spectrum of a number field, 
then $L(M,s)$ is the Dedekind zeta function $\z_k(s)$. 
The regulator map for $r=1$ is the classical regulator map
$$\sO_k^* \ot_\Z \Q \lra \prod_{v | \infty }\R$$
given by the logarithms of the absolute values for infinite places $v$. 
In fact, to obtain the isomorphism in this case, we need to add $\Q$ on the left-hand side which maps diagonally.
Then the conjecture reduces to the classical unit theorem of Dirichlet and the class number formula. 
The generalization to all $r \geq 1$ is due to Borel. 

Let $X_N$ be the Fermat curve of degree $N$ over a number field $k$ defined by the homogeneous equation
$$x_0^N+y_0^N=z_0^N.$$
The regulator map we study is:  
$$r_\sD \colon H^2_{\sM}(X_N,\Q(2))_\Z \lra H^2_{\sD}(X_{N,\R},\R(2)).$$
According to the conjecture, the dimension of the motivic cohomology group is 
$[k:\Q]$ times $\mathrm{genus}(X_N)=(N-1)(N-2)/2$. 
An element of the motivic cohomology group is given by a Milnor $K_2$-symbol on the function field. 
In \cite{ross}, Ross showed that the regulator image of the element
$$e_N=\{1-x,1-y\} \in H^2_{\sM}(X_N,\Q(2))_\Z,$$
where $x=x_0/z_0$, $y=y_0/z_0$ are the affine coordinates, 
is non-trivial. 
There are also relevant studies of Ross \cite{ross-cr} and Kimura \cite{k-kimura} (see \S 4.10, \S4.12). 
 
The corresponding $L$-function is $L(h^1(X_N),s)$ at $s=0$ (or at $s=2$ by the functional equation). 
Suppose for simplicity that $k$ contains $\mu_N$, the group of $N$-th roots of unity. 
As was proved by Weil \cite{weil-numbers}, \cite{weil-jacobi}, 
the $L$-function decomposes into the Jacobi-sum Hecke $L$-functions (see \S3.5) as
$$L(h^1(X_N),s)=\prod_{(a,b) \in I_N} L(j_N^{a,b},s)$$ 
where $I_N=\{(a,b) \mid a,b \in \Z/N\Z, a,b,a+b \neq 0\}$, 
hence satisfies all the desired analytic properties. 

In the study of Fermat varieties, the decomposition in the category of motives with 
coefficients is extremely useful. 
The terminology ``Fermat motive" appeared in Shioda's paper \cite{shioda}, 
although the idea goes back to the paper of Weil \cite{weil-jacobi} (see \S 3.7). 
When $\mu_N \subset k$, the group $\mu_N \times \mu_N$ acts on $X_N$ 
and using this action, we have a decomposition 
$$h^1(X_N) = \bigoplus_{(a,b) \in I_N} X_N^{a,b}$$
in the category $\sM_{k,\Q(\mu_N)}$ of motives over $k$ with coefficients in $\Q(\mu_N)$ (see \S2.8), 
which corresponds to the above decomposition of the $L$-function (Theorem \ref{fermat-L}). 

By projecting $e_N$, we obtain an element $$e_N^{a,b} \in H_\sM^2(X_N^{a,b},\Q)_\Z$$ for each $(a,b) \in I_N$. 
Our main result Theorem \ref{main-theorem} calculates the image of $e_N^{a,b}$ 
under each $v$-component 
$r_{\sD,v}$ of $r_\sD$, where $v$ is an infinite place of $k$. 
Define a  hypergeometric function of two variables by
$$F(\a,\b; x, y) = \sum_{m,n \geq 0} \frac{(\a,m)(\b,n)}{(\a+\b+1,m+n)} x^m y^n$$
where $(\a,m) = \a(\a+1)\cdots (\a+m-1)$ is the Pochhammer symbol. 
This is a special case of Appell's $F_3$, one of his four hypergeometric functions of two variables 
\cite{appell}. 
Then the regulator is expressed by special values  
$$\wt{F}(\a,\b) := \frac{\vG(\a)\vG(\b)}{\vG(\a+\b+1)}F(\a,\b;1,1)$$
for $\a, \b \in \frac{1}{N}\Z$, at the point $(x,y)=(1,1)$ which lies on the boundary of the domain of convergence. 
Notice that the period of $X_N^{a,b}$ is essentially the Beta value $B(\tfrac{a}{N},\tfrac{b}{N})$, 
whose inverse is related to a value of Gauss' hypergeometric function as  
$$ F(\a,\b,\a+\b+1;1) = \frac{\a+\b}{\a\b} B(\a,\b)^{-1}.$$
The value $\wt{F}(\a,\b)$ can also be written by the value at $x=1$ of 
Barnes' generalized hypergeometric function ${}_3F_2$ of one variable (see \S 4.10). 
We shall show the non-vanishing of $r_{\sD,v}(e_N^{a,b})$ by using the integral representation of $F_3$ (see \S 4.9).  
Since each target $H_\sD^2(X_{N,v}^{a,b},\R(2))$ is one-dimensional, 
we obtain the surjectivity of $r_{\sD,v}$ for each $X_N^{a,b}$, hence for $X_N$ (Corollary \ref{corollary-1}). 

For a general number field $k$ not necessarily containing $\mu_N$, we also have a motivic decomposition
$$h^1(X_N) = \bigoplus_{[a,b]_k} X_N^{[a,b]_k},$$
where $[a,b]_k$ is the orbit of $(a,b)$ by the action of 
$\Gal(k(\mu_N)/k) \subset (\Z/N\Z)^*$.  
The study of $X_N^{[a,b]_k}$ is essentially the same as that of $X_N^{a,b}$
and  we can also derive results for $X_N^{[a,b]_k}$.  

If $N=3$, $4$, $6$ and $k\subset\Q(\mu_N)$, then there is only one infinite place and we obtain 
the surjectivity of the whole $r_\sD$ (Corollary \ref{corollary-3}). 
In general, however, we do not have enough elements for the Beilinson conjecture.  
An attempt is to use the action of the symmetric group of degree $3$ acting on $X_N$ as permutations of the coordinates.  
In this way, we obtain at most $3$ elements from each $e_N^{a,b}$, and we shall show that 
they are actually enough for the surjectivity of the whole $r_{\sD}$ if $N=5, 7$ and $k\subset\Q(\mu_N)$,  
with a restriction on $(a,b)$ when $N=7$ (Theorem \ref{N=5}, Proposition \ref{N=7}). 

This paper is constructed as follows. 
In \S 2, we first recall briefly the necessary materials on motives and fix our notations. 
Then we define motives $X_N^{a,b}$, $X_N^{[a,b]_k}$ associated to Fermat curves 
and study the relations among them.  
In \S 3, after recalling the definition of the $L$-function of a motive with coefficients, 
we calculate the $L$-functions of our motives and derive basic properties.  
At the end, we compare our $L$-functions with the Artin $L$-functions of Weil. 
Finally in \S 4, we first recall the Beilinson regulator and the Beilinson conjecture for motives with coefficients. 
Then we define elements in the motivic cohomology groups and study the Deligne cohomology groups 
of our motives.   
The main results are stated in \S 4.7 and proved in \S 4.8, \S4.9 after introducing Appell's hypergeometric function.  
In \S 4.10, we introduce Barnes' hypergeometric function and discuss some variants. 
At the end, we calculate the action of the symmetric group and give applications for $N=5$ and $7$.   

\medskip
A part of this work was done when the author was visiting l'Institut de Math\'ematiques de Jussieu from 2004 to 2006, 
supported by the JSPS Postdoctoral Fellowships for Research Abroad. I would like to thank them for their  
hospitality and support. 
I would like to thank sincerely Bruno Kahn for his hospitality and enlightening discussions. 
Finally, I am grateful to Seidai Yasuda for valuable discussions.

\section{Fermat Motives}

\subsection{Motives}

We recall briefly the definition of the category of pure motives modulo rational equivalences 
(``Chow motives"). For more details, see \cite{scholl} and its references. 

For a field $k$, let $\sV_k$ be the category of smooth projective $k$-schemes. 
For $X \in \sV_k$ and a non-negative integer $n$, let $\CH^n(X)$ (resp. $\CH_n(X)$) 
be the Chow group 
of codimension-$n$ (resp. dimension-$n$) algebraic cycles on $X$ modulo rational equivalences. 
For example, $\CH^1(X)$ is the Picard group.  
Recall that for a flat (resp. proper) morphism 
$f \colon X \ra Y$, we have the pull-back (resp. push-forward) map 
$$f^*\colon \CH^n(Y)\lra \CH^n(X),\quad 
f_* \colon \CH_n(X) \lra \CH_n(Y).$$
If $f$ is flat and finite of degree $d$, we have $f_* \circ f^* = d$. 
In particular, $f^*$ (resp. $f_*$) is injective (resp. surjective) modulo torsion. 

For $X$, $Y \in \sV_k$, the group of {\em correspondences} of degree $r$ 
from $X$ to $Y$ is defined by
\begin{equation*}\begin{split}
\Corr^r(X,Y) & = \bigoplus_i \Q \ot_\Z \CH^{\dim X_i +r}(X_i \times Y)\\
& = \bigoplus_j \Q \ot_\Z \CH_{\dim Y_j -r}(X \times Y_j)
\end{split}\end{equation*}
where $X_i$ (resp. $Y_j$) are the irreducible components of $X$ (resp. $Y$). 
For a morphism $f\colon Y \ra X$, let $\vG_f \subset X \times Y$ be the transpose of its graph. 
Then it defines an element of $\Corr^0(X,Y)$.   
The composition of correspondences 
$$\Corr^r(X,Y) \ot \Corr^s(Y,Z) \lra \Corr^{r+s}(X,Z)$$ 
is defined by 
$$f \ot g \longmapsto g \circ f = {\pr_{XZ}}_*(\pr_{XY}^*(f) \cdot \pr_{YZ}^*(g)),$$
where $\pr_{**}$ is the projection from $X \times Y \times Z$ to the indicated factors, and $\cdot$ is the intersection product. In particular, we have 
$$\vG_f \circ \vG_g = \vG_{g \circ f}.$$ 
The class of the diagonal $\Delta_X = \vG_{\id_X}$ is the identity for the composition. 

The category $\sM_k = \sM_{k,\Q}$ of {\em motives over $k$ with $\Q$-coefficients} is defined as follows. 
The objects are triples $(X,p,m)$ where $X \in \sV_k$, 
$m \in \Z$, and 
$p \in \Corr^0(X,X)$ is an idempotent, that is, $p \circ p = p$. The morphisms are defined by 
\begin{equation*}
\Hom_{\sM_k}((X,p,m),(Y,q,n)) = q \circ \Corr^{n-m}(X,Y) \circ p.
\end{equation*}
We simply write $(X,p)$ instead of $(X,p,0)$, and $h(X)$ instead of $(X,\varDelta_X)$. 
Then, $h$ defines a {\em contravariant} functor 
\begin{equation*}
h \colon \sV_k^\mathrm{opp} \lra \sM_k; \quad X \longmapsto h(X), \ f \longmapsto \vG_f.
\end{equation*}
For a field extension $E$ of $\Q$, the category $\sM_{k,E}$ of {\em motives with $E$-coefficients}  is defined: it 
has the same objects as $\sM_k$, and the morphisms 
$$\Hom_{\sM_{k,E}}(M,N) = E \ot_\Q \Hom_{\sM_k}(M,N).$$
When $E \subset E'$, we regard $\sM_{k,E}$ as a subcategory of $\sM_{k,E'}$.  

The category $\sM_{k,E}$ is an additive $E$-linear category with the direct sum   
$$(X,p,m) \oplus (Y,q,n) = (X \sqcup Y, p \oplus q, m+n),$$ 
and the zero object $\mathbf{0}=h(\phi)$. 
It is a peudo-abelian category, that is, 
every projector (i.e. $f \in \End_{\sM_{k,E}}(M)$ such that $f \circ f =f$) has an image. 
For example, 
$$(X,p)=\Im(p\colon h(X) \ra h(X)).$$
On $\sM_{k,E}$, there exists a natural tensor product $\ot$ such that 
$$h(X) \ot h(Y) = h(X \times Y).$$
The identity for $\ot$ is the {\em unit motive}
$\mathbf{1} = h(\Spec k)$. 
The {\em Lefschetz motive} is defined by 
$\mathbf{L} = h(\Spec k, \varDelta_X,-1)$.
Then we have 
$$(X,p,m) \ot \mathbf{L}^{\ot n} = (X,p,m-n).$$
For a motive $M=(X,p,m)$ with $\dim X =d$, its {\em dual motive} 
is defined by  
\begin{equation*}
M^\vee = (X,{}^t\! p,d-m)
\end{equation*}
where ${}^t\! p$ is the transpose of $p$.  
For an integer $r$, the $r$-th {\em Tate twist} of $M$ is defined by
$$M(r)= (X,p,m+r) = M \ot \mathbf{L}^{\ot(-r)}.$$

For a morphism $f \colon X \ra Y$, we have the pull-back  
\begin{equation*}
f^* := \vG_f \colon h(Y) \lra h(X).
\end{equation*}
On the other hand, for irreducible $X$ and $Y$, 
we have the push-forward  
\begin{equation*}
f_* := \,^t \vG_f\colon h(X) \lra h(Y)(\dim Y - \dim X). 
\end{equation*}
Suppose that $X$, $X'$, $Y$, $Y' \in \sV_k$ are irreducible, and let $f \in \Corr^r(X,Y)$. Then, 
for morphisms $\a \colon X' \ra X$, $\b \colon Y' \ra Y$, we have
\begin{equation}\label{corr-pull}
\b^*  \circ f \circ \a_* =(\a \times \b)^* f, 
\end{equation}
and for morphisms $\a\colon X \ra X'$, $\b \colon Y \ra Y'$, we have
\begin{equation}\label{corr-push}
\b_*  \circ f \circ \a^* =(\a \times \b)_* f. 
\end{equation}
If $f\colon X \ra Y$ is a finite morphism of degree $d$, 
we have 
$f_* \colon h(X) \ra h(Y)$, and the formulae \eqref{corr-pull}, \eqref{corr-push} lead to: 
\begin{equation*}
f^* \circ f_* = [X \times_Y X]  \ \in \End(h(X))
\end{equation*}
and
\begin{equation*}
f_* \circ f^* = d [\Delta_Y] \ \in \End(h(Y)).
\end{equation*}
In particular, $f^*$ (resp. $f_*$) is injective (resp. surjective). 

\subsection{Motives of curves}

In the case of curves, we have the so-called Chow-K\"unneth decomposition, which is still conjectural in general. 

Let $f \colon X \ra \Spec k$ be a smooth irreducible projective curve, 
and suppose for simplicity that it has a $k$-rational point $x$. 
Define correspondences $e^i \in \Corr^0(X,X)$ by  
\begin{equation*}
e^0 = \{x\} \times X, \quad e^2 = X \times \{x\}, \quad e^1= \varDelta_X - e^0 -e^2.
\end{equation*}
One sees easily that $e^0$, $e^2$ are idempotents and that $e^0 \circ e^2 = e^2 \circ e^0 =0$, 
hence $e^1$ is also an idempotent orthogonal to $e^0$ and $e^2$. 
The $i$-th cohomological motive of $X$ defined by 
\begin{equation*}
h^i(X) =(X,e^i).
\end{equation*}
By definition, we have a decomposition (depending on the choice of $x$)  
$$h(X) \simeq h^{0}(X)\oplus h^{1}(X) \oplus h^{2}(X).$$

Since $e^0=f^* \circ x^*$, the map $f^*\colon \mathbf{1} \ra h(X)$ induces an isomorphism 
$\mathbf{1} \os{\simeq}{\lra} h^0(X)$ whose inverse is given by $x^*$. 
Similarly, since $e^2=x_* \circ f_*$, the map $f_*\colon h(X) \ra \mathbf{L}$ induces an isomorphism 
$h^2(X) \os{\simeq}{\lra} \mathbf{L}$ whose inverse is given by $x_*$. 
If $X = \P^1$, we have $h(\P^1) \simeq \mathbf{1} \oplus \mathbf{L}$. 

\subsection{Functorialities}

Let $K/k$ be a field extension, and 
$$\varphi_{K/k}\colon \Spec K \lra \Spec k$$
be the structure morphism. 
Then we have the ``scalar extension"  functor
\begin{equation*}
\sV_k \lra \sV_K; \quad X \longmapsto X_K := X \times_k  K. 
\end{equation*}
The pull-back on the Chow group $\CH^*(X \times_k Y) \ra \CH^*(X_K \times_K Y_K)$ 
defines a homomorphism 
$$\Corr^r(X,Y) \lra \Corr^r(X_K,Y_K); \quad f \longmapsto f_K,$$
which is injective. Therefore, the above functor extends to a faithful functor 
\begin{equation*}
{\varphi_{K/k}^*}\colon \sM_{k,E} \lra \sM_{K,E}; \quad (X,p,m) \longmapsto (X_K, p_K,m)
\end{equation*}

On the other hand, for a finite separable extension $K/k$,  
we have Grothendieck's ``scalar restriction" functor 
\begin{equation*}
\sV_K \lra \sV_k
\end{equation*}
which sends $X \ra \Spec K$ to the composite 
$$X_{|k} := X \ra \Spec K \ra \Spec k.$$
The push-forward
$\CH^*(X \times_K Y) \ra \CH^*(X_{|k} \times_k Y_{|k})$
defines a homomorphsim  
$$\Corr^r(X,Y) \lra \Corr^r(X_{|k},Y_{|k}); \quad f \longmapsto f_{|k}$$
and induces a functor
\begin{equation*}
\varphi_{K/k *}\colon \sM_{K,E} \lra \sM_{k,E}; \quad (X,p,m) \longmapsto (X_{|k},p_{|k},m),
\end{equation*}
which is left and right adjoint to $\vphi_{K/k}^*$. 

\subsection{Fermat curves}
Let $k$ be a field and $N$ be a positive integer prime to $\mathrm{char}(k)$. 
Let $X_N=X_{N,k}$ be the smooth projective curve over $k$ 
defined by the homogeneous equation 
\begin{equation}\label{equation-fermat} 
x_{0}^N+y_{0}^N=z_{0}^N.
\end{equation}
It has genus $(N-1)(N-2)/2$. 
Define a closed subscheme by 
$$Z_N = X_N \cap \{z_0=0\}$$ 
and let $U_N = X_N - Z_N$ be the open complement. 
The affine equation is written as 
\begin{equation*}
x^{N}+y^{N}=1 \quad (x=x_{0}/z_{0}, y=y_{0}/z_{0}).
\end{equation*}

If $N'$ divides $N$, with $N=N'd$, we have a $k$-morphism
\begin{equation*}
\pi_{N/N',k} \colon X_{N} \lra X_{N'}; \quad 
(x_{0}:y_{0}:z_{0}) \longmapsto (x^{d}_{0}:y^{d}_{0}:z^{d}_{0}) 
\end{equation*}
which is finite of degree $d^2$. It respecs $Z_*$, $U_*$, and \'etale over $U_{N'}$. 
On the other hand, for a field extention $K/k$, we have a canonical morphism
$$\pi_{N,K/k} \colon X_{N,K} \lra X_{N,k}. $$
We denote the composition as 
$$\pi_{N/N',K/k} \colon X_{N,K} \lra X_{N',k}.$$
By the evident relation 
$\pi_{N'/N'', K/k} \circ \pi_{N/N',L/K} =\pi_{N/N'', L/k}$, 
the curves $X_{N,k}$ for various $N$ and $k$ form a projective system. 

\subsection{Group actions}

Fix an algebraic closure $\ol k$ of $k$.  For $N$ prime to $\mathrm{char}(k)$, put 
\begin{equation*}
K_N=k(\mu_N).
\end{equation*}
For each $N$, fix a primitive $N$-th root of unity $\zeta_N \in K_N$ in such a way that 
$\z_{Nd}^d=\z_N$. 

Define finite groups by 
\begin{equation*}
G_N=\Z/N\Z \oplus \Z/N\Z, \quad H_N=(\Z/N\Z)^*. 
\end{equation*}
Then, $H_N$ acts (from the left) on $G_N$ by the multiplication on both factors  
and we put
\begin{equation*}
\vG_N=G_N \rtimes H_N.
\end{equation*}
We denote an element $(r,s) \in G_N$ also by $g_N^{r,s}$, and write the addition multiplicatively, i.e. 
$$g_N^{r,s} g_N^{r',s'} = g_N^{r+r',s+s'}.$$  
Let $H_{N,k} \subset H_N$ be the image of the injective homomorphism
$\Gal(K_N/k) \ra H_N$
which maps $\sigma$ to the unique element $h$ such that $\sigma(\zeta_N)=\zeta_N^h$. 
Finally, define a subgroup of $\vG_N$ by
\begin{equation*}
\vG_{N,k}= G_{N} \rtimes H_{N,k}. 
\end{equation*}

Now we define an action of $\vG_{N,k}$ on $X_{N,K_N}$. 
Throughout this paper, we let groups act on schemes from the right, so that they induce right actions 
on rational points, homology groups, etc. and left actions on functions, differential forms, 
cohomology groups, etc. 

First, let $G_N$ act on $X_{K_N}$ by 
\begin{equation*}
g_N^{r,s}(x_0:y_0:z_0) = (\z_N^r x_0: \z_N^s y_0:z_0).
\end{equation*}
Secondly, the action of $H_{N,k}=\Gal(K_N/k)$ on $K_N$ induces an action on $\Spec K_N$,   
hence, by the base-change, an action on $X_{N',K_N}$ for any $N'$. 
Finally, since the above actions satisfy 
$hg = h(g)h$ for any $g \in G_N$, $h \in H_{N,k}$, 
they extend to an action of $\vG_{N,k}$ on $X_{N,K_N}$. 
Summarizing, we have the following commutative diagram with the indicated automorphism groups: 
\begin{equation}\label{diagram-1}
\xymatrix{
X_{N,K_N} \ar[rr]^{G_N} \ar[d]_{H_{N,k}} \ar[rrd]^{\vG_{N,k}} && X_{1,K_N} \ar[d]^{H_{N,k}}\\
X_{N,k}\ar[rr] && X_{1,k}.
}\end{equation}

For $N'|N$, the canonical surjective homomorphisms
$$G_{N} \lra G_{N'} , \quad H_{N,k} \lra H_{N',k}, \quad \vG_{N,k} \lra \vG_{N',k}$$ 
are compatible with the morphisms 
$$\pi_{N/N',K_N}, \quad \pi_{N,K_{N}/K_{N'}}, \quad \pi_{K_{N}/K_{N'},N/N'},$$
respectively. 
We shall use the notation 
$$G_{N/N'} := \Ker(G_{N} \ra G_{N'}) =\Aut(X_{N,K_N}/X_{N',K_N}).$$
 
\subsection{Index sets}

We say that an element $(a,b) \in G_N$ is {\em primitive} if $\gcd(N,a,b)=1$, and 
let $G_N^\prim \subset G_N$ be the subset of primitive elements. 
If we put
$$d=\gcd(N,a,b), \quad N'=N/d, \quad a'=a/d, \quad b'=b/d, $$
then $(a',b') \in G_{N'}^\prim$. 
For $(a,b) \in G_{N}$, let $[a,b]_{k}$ denote its $H_{N,k}$-orbit. 
Then the map 
$$G_{N'} \lra G_N ; \quad (a',b') \mapsto (a'd,b'd)$$ 
induces bijections
$$G_N \simeq \bigsqcup_{N' | N} G_{N'}^\prim, \quad
H_{N,k}\backslash G_N \simeq \bigsqcup_{N' \mid N} H_{N',k}\backslash G_{N'}^\prim. $$
Since it induces a bijection 
$[a',b']_k \os{\simeq}{\ra}  [a,b]_k$, we have
$$\sharp[a,b]_{k}=\sharp[a',b']_{k}= \sharp H_{N',k}=[K_{N'}:k].$$

Define a subset of $G_N$ by  
$$I_N=\left\{(a,b) \in G_{N}  \mid a, b, a+b \neq 0\right\} 
$$
and put $I_N^\prim = I_N \cap G_N^\prim$. 
Then, $I_N$ and $I_N^\prim$ are stable under the action of $H_{N,k}$. 
Note that  
$\sharp I_N = (N-1)(N-2)$, twice the genus of $X_N$.  
We have also bijections
$$I_N \simeq \bigsqcup_{N' \mid N} I_{N'}^\prim,  \quad 
H_{N,k}\backslash I_N \simeq \bigsqcup_{N' \mid N} H_{N',k}\backslash I_{N'}^\prim. $$

\subsection{Projectors}

For an integer $N$, put  
\begin{equation*}
E_N=\Q(\mu_N) \subset \ol\Q 
\end{equation*}
and for each $N$, fix a primitive $N$-th root of unity $\xi_N \in \ol \Q$ in such a way that 
$\xi_{Nd}^d=\xi_N$.  

\begin{definition}
For $(a,b) \in G_{N}$, let $\theta_N^{a,b} \colon G_N \ra E_N^*$ be the character  defined by 
\begin{equation*}
\theta_{N}^{a,b}(g_N^{r,s}) = \xi_{N}^{ar+bs}. 
\end{equation*}
Any character of $G_N$ is uniquely written in this form.
If $N=N'd$, $a=a'd$ and $b=b'd$, then $\theta_N^{a,b}$ is the pull-back of $\theta_{N'}^{a',b'}$ by 
the natural homomorphism $G_N \ra G_{N'}$. 
\end{definition}

\begin{definition}
For $(a,b) \in G_N$, define an element of the group ring $E_N[G_N]$ by  
\begin{equation*}
p_N^{a,b} = \frac{1}{N^2} \sum_{g \in G_N} \theta_N^{a,b}(g^{-1}) g.
\end{equation*}
Evidently, 
\begin{equation}\label{property-p^{a,b}}
\sum_{(a,b)\in G_{N}}p_{N}^{a,b} =1, \quad
p_{N}^{a,b} p_{N}^{c,d} = \begin{cases} 
	p_{N}^{a,b} & \text{if $(a,b)=(c,d)$,} \\
	0 & \text{otherwise.} \end{cases} \\ 
\end{equation}
\end{definition}

\begin{definition}
For a class $[a,b]_k \in H_{N,k} \backslash G_N $, define an element of $E_N[G_{N}]$ by  
\begin{equation*}
p_{N}^{[a,b]_k} = \sum_{(c,d) \in [a,b]_k} p_{N}^{c,d}. 
\end{equation*}
Then we have 
\begin{equation}\label{property-p^{[a,b]}}
\sum_{[a,b]_k\in H_{N,k} \backslash G_N} p_{N}^{[a,b]_k}  =1, \quad
p_{N}^{[a,b]_k} p_{N}^{[c,d]_k}   = \begin{cases} 
	p_{N}^{[a,b]_k} &  \text{if $[a,b]_k=[c,d]_k$,} \\
	0 & \text{otherwise.} \end{cases} \
\end{equation}
\end{definition}

It is easy to prove: 
\begin{lemma}\label{projector-level-change} 
Let $N=N'd$. Under the natural homomorphism 
$E_N[G_N] \ra E_N[G_{N'}]$, 
\begin{enumerate}
\item 
$p_N^{a,b} \longmapsto\begin{cases}
p_{N'}^{a',b'} & \text{if $(a,b)=(a'd,b'd)$ for some $(a',b') \in G_{N'}$},\\
0 & \text{otherwise}, 
\end{cases}
$
\item 
$p_N^{[a,b]_k} \longmapsto 
\begin{cases}
p_{N'}^{[a',b']_k} & 
\text{if $[a,b]_k=[a'd,b'd]_k$ for some $[a',b']_k\in H_{N,k}\backslash G_{N'}$},\\
0 & \text{otherwise}.
\end{cases}
$
\end{enumerate}
\end{lemma}

\begin{definition}
Let $E_{N,k}$ be the subfield of $E_N$ fixed by $H_{N,k}$ viewed as a subgroup of $\Gal(E_N/\Q) \simeq H_N$.  
\end{definition}

Extend by linearity the action of $H_{N,k}$ on $G_N$ to the group ring 
$E_N[G_N]$ (notice: $H_{N,k}$ does not act on $E_N$). 
\begin{lemma}\label{projector-smaller}
For any $(a,b) \in G_N$ and $h \in H_{N,k}$, we have{\rm :} 
\begin{enumerate}
\item $h(p_N^{a,b})=p_N^{h^{-1}a,h^{-1}b}$, 
\item $h(p_N^{[a,b]_k})=p_N^{[a,b]_k}$, 
\item $p_N^{[a,b]_k}\in E_{N,k}[G_N]$.  
\end{enumerate}
\end{lemma}

\begin{proof}
(i) is easy and (ii) follows from (i). (iii) follows from 
$$
p_N^{[a,b]_k} 
= \frac{\sharp[a,b]_k}{\sharp H_{N,k}} \sum_{h \in H_{N,k}} p_N^{ha,hb} =\frac{\sharp[a,b]_k}{\sharp H_{N,k}} \frac{1}{N^2}
 \sum_{g \in G_N} \Tr_{E_N/E_{N,k}}(\theta_N^{a,b}(g^{-1})) g. 
$$
\end{proof}

\begin{remark}
In fact, $p_N^{[a,b]_k}\in E_{N',k}[G_N]$ where $N'=\mathrm{gcd}(N,a,b)$.
\end{remark}

\subsection{Fermat motives}

As a base point, we choose  
$$x=(0 \colon 1 \colon1) \  \in X_N(k)$$ 
so that it is compatible under the morphisms $\pi_{N/N',K/K'}$ for various degrees and base fields. 

The action of $G_N$ on $X_{N,K_N}$ induces an action (from the left) 
on $h(X_{N,K_N})$, and by linearity we obtain an $E_N$-algebra homomorphism
$$E_N[G_N] \lra \End_{\sM_{K_N,E_N}}(h(X_{N,K_N})). $$
By abuse of notation, the image of an element of the group ring will be denoted by the same letter. 
For example, we just denote by $g$ instead of $g^*$ or $\vG_g$. 

\begin{definition}
For $(a,b) \in G_N$, define{\rm :} 
$$X_{N}^{a,b} = (X_{N,K_N}, p_{N}^{a,b})\ \in \sM_{K_N,E_N}.$$
Then, by \eqref{property-p^{a,b}}, we have a decomposition 
\begin{equation*}
h(X_{N,K_N}) \simeq  \bigoplus_{(a,b)\in G_{N}} X_{N}^{a,b}. 
\end{equation*}
\end{definition}

\begin{proposition}\label{decomposition-X_K}
We have isomorphisms in  $\sM_{K_N,E_N}${\rm :}
\begin{enumerate}
	\item $X_N^{0,0} \simeq \mathbf{1} \oplus \mathbf{L}$, 
	\item $X_N^{a,b} \simeq {\bf 0}$ \/ if only one of $a$, $b$, $a+b$ is \/ $0$,  
	\item $h^1(X_{N,K_N}) \simeq \bigoplus_{(a,b) \in I_N} X_N^{a,b}$. 
\end{enumerate}
\end{proposition}

\begin{proof}
Let $f \colon X_{N,K_N} \ra \Spec K_N$ be the structure morphism. 
For $g \in G_N$, we have 
$$e^0 \circ g = \vG_{x \circ f} \circ \vG_g = \vG_{g \circ x \circ f} = \vG_{g(x) \circ f},  $$
the graph of the constant morphism with value $g(x)$. 
Since 
$$g_N^{r,s}(x)=(0\colon \z_N^s\colon 1)=g_N^{0,s}(x),$$ 
we obtain
\begin{equation}\label{e^0}
e^0 \circ p_N^{a,b} 
= \frac{1}{N^2} \left(\sum_r \xi_N^{-ar}\right) \left(\sum_{s} \xi_N^{-bs} \vG_{g_N^{0,s}(x) \circ f}\right). 
\end{equation}
In particular, we have 
\begin{align*}
e^0-e^0 \circ p_N^{0,0} = \frac{1}{N} \Bigl(N\vG_{x\circ f} - \sum_s \vG_{g_N^{0,s}(x) \circ f}\Bigr)
= \frac{1}{N} \div\left(\frac{1-y_1}{x_1}\right)=0 
\end{align*}
where $(x_i,y_i)$ ($i=1,2$) are the affine coordinates of the $i$-th component of 
$X_{N,K_N} \times X_{N,K_N}$.  
Since 
$e^2 \circ g = e^2$
for any $g \in G_N$, we have
\begin{equation}\label{e^2}
e^2 \circ p_N^{a,b} = \frac{1}{N^2} \left(\sum_{r,s} \xi_N^{-(ar+bs)}\right) e^2. 
\end{equation}
In particular, we have
$e^2 \circ p_N^{0,0} = e^2$. 
Therefore, we have
\begin{multline*}
e^1 \circ p_N^{0,0} = (1 - e^0 - e^2) \circ p_N^{0,0} \\
= p_N^{0,0} - e^0 - e^2
= \frac{1}{N^2} \div\left(\frac{y_1^N-y_2^N}{(1-y_1)^N (1-y_2)^N}\right)=0, 
\end{multline*} 
hence (i) is proved.

If $a \neq 0$ and $b=0$, then we have
$$p_N^{a,0} = \frac{1}{N^2} \sum_r \xi^{-ar} \sum_s \vG_{g_N^{r,s}}.$$ 
Since
$$\sum_s\vG_{g_N^{r,s}} - \sum_s\vG_{g_N^{r',s}} 
= \div\left(\frac{x_1-\z_N^rx_2}{x_1-\z_N^{r'}x_2}\right)=0, $$
$\sum_s \vG_{g_N^{r,s}}$ does not depend on $r$, hence we obtain $p_N^{a,0} =0$. 
The other cases of (ii) are similarly proved. 

Finally, if $(a,b) \in I_N$, then by \eqref{e^0} and \eqref{e^2}, 
we have $e^0\circ p_N^{a,b} = e^2 \circ p_N^{a,b} =0$, 
hence
\begin{align*}
\sum_{(a,b) \in I_N} p_N^{a,b} = e^1 \circ \sum_{(a,b) \in I_N} p_N^{a,b} 
= e^1 \circ \sum_{(a,b) \in G_N} p_N^{a,b} = e^1, 
\end{align*}and (iii) follows. 
\end{proof}

Now, we define Fermat motives over $k$. 
For any $E$, an element of $E[G_N]$ fixed by the action of $H_{N,k}$ 
defines a cycle on $X_{N,K_{N,k}} \times_{K_N} X_{N,K_N}$ {\em defined over $k$}, 
i.e. in the image of 
$$
(\pi_N\times_{K_N}\pi_N)^*\colon E \ot_\Z \CH^1(X_N \times_k X_N) 
\lra 
E \ot_\Z \CH^1(X_{N,K_N} \times_{K_N} X_{N,K_N}),
$$
where $\pi_N = \pi_{N,K_N/k}$. To make the situation clear, we denote the canonical morphisms as:
\begin{equation}\label{morphism-product}
\xymatrix{
X_{N,K_N} \times_{K_N} X_{N,K_N} \ar@{^{(}->}[r] \ar[rd]_{\pi_N\times_{K_N} \pi_N \ \ } & X_{N,K_N} \times_k X_{N,K_N}\ar[d]^{\pi_N\times_k\pi_N} \\
& X_N \times_k X_N . 
}
\end{equation}
Note that $\pi_N\times_{K_N} \pi_N$ and $\pi_N\times_{k} \pi_N$ 
are finite morphisms of degree $\sharp H_{N,k}$
and $\sharp H_{N,k}^2$, respectively. 
In particular, $(\pi_N\times_{K_N}\pi_N)^*$ is injective and its left-inverse is given by $(\sharp H_{N,k})^{-1} (\pi_N\times_{K_N}\pi_N)_*$. 
Since the intersection product is compatible with the pull-back, we obtain 
an $E$-algebra homomorphism
$$E[G_N]^{H_{N,k}} \lra E \ot_\Z \CH^1(X_{N} \times_k X_N) = \End_{\sM_{k,E}}(h(X_N)).$$

By Lemma \ref{projector-smaller} (iii), $p_N^{[a,b]_k}$ defines an element of 
$\End_{\sM_{k,E_{N,k}}} (h(X_N))$, which we also denote by the same letter.   
Since 
$$(\pi_N\times_{K_N}\pi_N)^*(\pi_N\times_{K_N}\pi_N)_* p_N^{a,b}= \sum_{h\in H_{N,k}} p_N^{ha,hb},$$
we have
\begin{equation}\label{definition-p_N^{[a,b]_k}}
(\pi_N\times_{K_N}\pi_N)_* p_N^{a,b}= \frac{\sharp H_{N,k}}{\sharp[a,b]_k} \, p_N^{[a,b]_k}.
\end{equation}

\begin{definition}
For $[a,b]_k \in H_{N,k} \backslash G_N$,  define{\rm :}  
$$X_N^{[a,b]_k} = (X_{N}, p_{N}^{[a,b]_k}) \ \in \sM_{k, E_{N,k}}.$$
Then, by \eqref{property-p^{[a,b]}}, we have a decomposition
\begin{equation*}
 h(X_{N}) \simeq \bigoplus_{[a,b]_k \in H_{N,k} \backslash G_N} X_{N}^{[a.b]_k}.
\end{equation*}
\end{definition}

\begin{proposition}\label{decomposition-X} 
We have isomorphisms in $\sM_{k,E_{N,k}}${\rm :}
\begin{enumerate}
	\item $X_N^{[0,0]_k} \simeq \mathbf{1} \oplus \mathbf{L}$, 
	\item $X_N^{[a,b]_k} \simeq  {\bf 0}$ \ if only one of $a$, $b$, $a+b$ is $0$,  
	\item $h^1(X_N) \simeq \bigoplus_{[a,b]_k \in H_{N,k}\backslash I_N} X_N^{[a,b]_k}$. 
\end{enumerate}
\end{proposition}

\begin{proof}
Since 
$$\End_{\sM_{k,E_{N,k}}}(h(X_N)) \lra \End_{\sM_{K_N,E_N}}(h(X_{N,K_N}))$$
is injective, 
we can compute $e^i \circ p_N^{[a,b]_k}$ in the latter ring. 
Then the proof reduces to 
Proposition \ref{decomposition-X_K}. 
\end{proof}
 
\begin{remark}\label{fermat-dual}
Since $\,^tp_N^{a,b} = p_N^{-a,-b}$, we have 
\begin{equation*}
(X_N^{a,b})^\vee = X_N^{-a,-b}(1), \quad (X_N^{[a,b]_k})^\vee = X_N^{[-a,-b]_k}(1). 
\end{equation*}
\end{remark}

\begin{remark}\label{C_N}
We can also define similarly a motive $X_N^{[a,b]}$
for each $H_N$-orbit $[a,b]$ of $(a,b)\in G_N$. 
Then one shows that $X_N^{[a,b]} \in \sM_{k}$. 
If $\Gal(K_N/k)=H_N$,  e.g. $k=\Q$, then $X_N^{[a,b]_k} =X_N^{[a,b]}$. 
For integrs $0<a,b<N$, let $C_N^{a,b}$ be the projective smooth curve whose affine 
equation is given by 
\begin{equation*}
v^N=u^a(1-u)^b 
\end{equation*}
(it has singularities possibly at $u=0,1,\infty$).  There exists a morphism 
\begin{equation*}
 \psi \colon X_N \lra C_N^{a,b}; \quad (x,y) \longmapsto (u,v)=(x^N,x^ay^b). 
\end{equation*}
Suppose that $N$ is a prime number and  $(a,b) \in I_N$. 
Then one shows that $\psi$ induces an isomorphism in $\sM_{k}$: 
\begin{equation*}
X_N^{[a,b]} \simeq h^1(C_N^{a,b}).
\end{equation*}
\end{remark}

\subsection{Relations among Fermat motives}

When $(a,b)$ is not primitive, our motives $X_N^{a,b}$, $X_N^{[a,b]_k}$ 
come from motives of lower degree. 
Let $N=N'd$ and use the following abbreviated notations:
\begin{equation*}
\xymatrix{
X_{N,K_N} \ar[dd]_{\pi_{N}} \ar[r]^{\pi_{K_N}} & X_{N',K_N} \ar[d]^{\pi_{K_N/K_{N'}}}  \\
& X_{N',K_{N'}} \ar[d]^{\pi_{N'}}\\
X_{N} \ar[r]^{\pi_k} & X_{N'}. 
}\end{equation*}
Consider the homomorphisms on Chow groups with coefficients in $E_N$ (resp. $E_{N,k}$) 
induced by the $K_N$-morphism (resp. $k$-morphism) 
$$
\pi_{K_N}\times\pi_{K_N} \colon X_{N,K_N} \times X_{N,K_N}  \lra X_{N',K_N} \times X_{N',K_N}, 
$$
resp. 
$$
\pi_{k}\times\pi_{k} \colon X_{N} \times X_{N}  \lra X_{N'} \times X_{N'}.
$$

\begin{lemma}\label{projector-functoriality} 
Let the notations as above. 
\begin{enumerate}
\item For $(a,b) \in G_N$, we have
$$
(\pi_{K_N}\times\pi_{K_N})_*p_{N}^{a,b}  = 
\begin{cases}
d^2 p_{N',K_N}^{a',b'} & \text{if $(a,b)=(a'd,b'd)$ for some $(a',b')\in G_{N'}$,}\\
0 & \text{otherwise}.
\end{cases}
$$
\item For $(a',b') \in G_{N'}$, we have 
$(\pi_{K_N}\times\pi_{K_N})^*p_{N',K_N}^{a',b'} = d^2 p_N^{a'd,b'd}$. 
\item For $(a,b) \in G_N$, we have
$$
(\pi_k\times\pi_k)_*p_{N}^{[a,b]_k} = 
\begin{cases}
d^2 p_{N'}^{[a',b']_k} & \text{if $[a,b]_k=[a'd,b'd]_k$ for some $(a',b')\in G_{N'}$,}\\
0 & \text{otherwise}.
\end{cases}
$$
\item For $(a',b') \in G_{N'}$, we have 
$(\pi_k\times\pi_k)^*p_{N'}^{[a',b']_k} = d^2 p_N^{[a'd,b'd]_k}$. 
\end{enumerate}
\end{lemma}

\begin{proof}
Since the degree of $X_{N,K_N}$ over $X_{N',K_N}$ is $d^2$, we have 
$$(\pi_{K_N}\times\pi_{K_N})_*\vG_{g_N^{r,s}} = d^2 \vG_{g_{N'}^{r,s}},$$ 
and (i) follows from Lemma \ref{projector-level-change} (i). 
On the other hand, (ii) follows easily from  
$$(\pi_{K_N}\times\pi_{K_N})^* \vG_{g_{N'}^{r,s}} 
= \sum_{g_N^{r,s} \mapsto g_{N'}^{r',s'}} \vG_{g_N^{r,s}}.$$ 

Put $m=\sharp H_{N,k} / \sharp[a,b]_k$. Then, using similar notations as \eqref{morphism-product}, 
we have:  
\begin{align*}
&(\pi_k\times\pi_k)_*p_{N}^{[a,b]_k} \\
&= m^{-1} (\pi_k\times\pi_k)_* (\pi_N\times_{K_N}\pi_N)_* p_{N}^{a,b} \\
&= m^{-1} (\pi_{N'}\times_{K_{N'}}\pi_{N'})_* (\pi_{K_N/K_{N'}}\times_{K_N}\pi_{K_N/K_{N'}})_*
(\pi_{K_N}\times\pi_{K_N})_* p_{N}^{a,b} \\
& \us{{\rm (i)}}{=} \begin{cases}
m^{-1}d^2   (\pi_{N'}\times_{K_{N'}}\pi_{N'})_* (\pi_{K_N/K_{N'}}\times_{K_N}\pi_{K_N/K_{N'}})_* p_{N',K_N}^{a',b'}, \\
0.
\end{cases}
\end{align*}
Using 
$(\pi_{K_N/K_{N'}}\times_{K_N}\pi_{K_N/K_{N'}})_* p_{N',K_N}^{a',b'} = [K_N:K_{N'}] p_{N'}^{a',b'}$, 
\eqref{definition-p_N^{[a,b]_k}} and $\sharp [a,b]_k = \sharp [a',b']_k$, 
we obtain (iii). 

Finally, (iv) follows from the injectivity of $(\pi_N\times_{K_N}\pi_N)^*$ and
\begin{align*}
&(\pi_N\times_{K_N}\pi_N)^*(\pi_k\times\pi_k)^*p_{N'}^{[a',b']_k} \\
&= (\pi_{K_N}\times\pi_{K_N})^*(\pi_{K_N/K_{N'}}\times_{K_N}\pi_{K_N/K_{N'}})^*(\pi_{N'}\times_{K_{N'}}\pi_{N'})^*p_{N'}^{[a',b']_k} \\
& = \sum_{(e',f') \in [a',b']_k}  (\pi_{K_N}\times\pi_{K_N})^*p_{N',K_N}^{e',f'} 
\us{{\rm (ii)}}{=} \sum_{(e',f') \in [a',b']_k} d^2 p_N^{e'd,f'd} \\
& = d^2 \sum_{(e,f) \in [a'd,b'd]_k} p_N^{e,f} 
 = d^2 (\pi_N\times_{K_N}\pi_N)^* p_N^{[a,b]_k}. 
\end{align*}
\end{proof}

\begin{proposition}\label{prop-level}
Let $N=N'd$, $(a',b') \in G_{N'}$ and $(a,b) = (a'd, b'd) \in G_N$. 
Then we have{\rm :} 
\begin{enumerate} 
\item $X_N^{a,b} \simeq \varphi_{K_N/K_{N'}}^* X_{N'}^{a',b'}$ \  in $\sM_{K_N,E_N}$, 
\item $X_N^{[a,b]_k} \simeq X_{N'}^{[a',b']_k}$ \ in $\sM_{k,E_{N,k}}$. 
\end{enumerate}
\end{proposition}

\begin{proof}
(i) By definition, $\varphi_{K_N/K_{N'}}^* X_{N'}^{a',b'} = (X_{N',K_N}, p_{N',K_N}^{a',b'})$. 
Consider the following commutative diagram with $\pi=\pi_{K_N}$: 
\begin{equation}\label{four-1}
\xymatrix{
h(X_{N',K_N})\ar[r]^{\pi^*} \ar[d]_{p_{N',K_N}^{a',b'}} & 
h(X_{N,K_N}) \ar[r]^{\pi_*} \ar[d]_{p_N^{a,b}} & 
h(X_{N',K_N}) \ar[r]^{\pi^*} \ar[d]_{p_{N',K_N}^{a',b'}} & 
h(X_{N,K_N})  \ar[d]_{p_N^{a,b}} \\
h(X_{N',K_N})\ar[r]^{\pi^*} & h(X_{N,K_N}) \ar[r]^{\pi_*}  & h(X_{N',K_N})  \ar[r]^{\pi^*}  & h(X_{N,K_N}).  
}\end{equation}
Let us show that the commutativity of the first square.  
Since $\circ \pi_* = (\pi \times \id)^*$ by \eqref{corr-pull} and is injective, 
it suffices to show the commutativity after applying it. 
First, we have
$$\pi^*\circ p_{N',K_N}^{a',b'} \circ \pi_* =(\pi \times \pi)^* p_{N',K_N}^{a',b'} = d^2p_N^{a,b}$$
by Lemma \ref{projector-functoriality} (ii). 
On the other hand, using \eqref{corr-pull} and \eqref{corr-push}, we have 
$$p_N^{a,b} \circ \pi^* \circ \pi_* = (\pi \times \id)^*(\pi \times \id)_*p_N^{a,b}
= p_N^{a,b} \circ \sum_{g\in G_{N/N'}} g .$$ 
Since $\theta_N^{a,b}(g)=1$ for $g \in G_{N/N'}$, we have $p_N^{a,b} \circ g= p_N^{a,b}$, 
hence the commutativity is proved.  
The commutativity of the second square is the ``transpose" of the first one: 
$$\pi_* \circ p_N^{a,b}={}^t\!(p_N^{-a,-b} \circ \pi^* )
= {}^t\!(\pi^*\circ p_{N',K_N}^{-a',-b'}) =p_{N',K_N}^{a',b'} \circ \pi_*.$$ 
Therefore, $\pi^*$ maps $X_{N',K_N}^{a',b'}$ to $X_N^{a,b}$ and 
$\pi_*$ maps $X_N^{a,b}$ to $X_{N',K_N}^{a',b'}$. 
Recall that $\pi_* \circ \pi^* = d^2$. On the other hand, we have 
$$\pi^*\circ\pi_*\circ p_N^{a,b}= \pi^*\circ p_{N',K_N}^{a',b'}\circ \pi_* = d^2 p_N^{a,b}.$$ 
Therefore, 
$$\pi^*\colon X_{N',K_N}^{a',b'} \lra X_N^{a,b}, \quad d^{-2}\pi_*\colon X_N^{a,b} \lra X_{N',K_N}^{a',b'}$$ 
are isomorphisms inverse to each other.  

(ii) Consider the commutative diagram with $\pi=\pi_{k}$: 
\begin{equation}\label{four-2}
\xymatrix{
h(X_{N'})\ar[r]^{\pi^*} \ar[d]_{p_{N'}^{[a',b']_k}} & h(X_{N}) \ar[r]^{\pi_*} \ar[d]_{p_N^{[a,b]_k}} 
& h(X_{N'}) \ar[r]^{\pi^*} \ar[d]_{p_{N'}^{[a',b']_k}} & h(X_{N})  \ar[d]_{p_N^{[a,b]_k}} \\
h(X_{N'})\ar[r]^{\pi^*} & h(X_{N}) \ar[r]^{\pi_*}  & h(X_{N'})  \ar[r]^{\pi^*}  & h(X_{N}). 
}\end{equation}
Similarly as above using Lemma \ref{projector-functoriality} (iv), the commutativity of the first square is reduced to show
\begin{equation*}\label{p-pi-pi}
p_N^{[a,b]_k} \circ \pi^* \circ \pi_* = d^2 p_N^{[a,b]_k}.
\end{equation*}
We can compute the composition after applying the faithful functor $\vphi_{K_N/k}^*$. 
Then we are reduced to the calculation of (i). The second square is again the ``transpose" of the first one.   The rest of the proof is parallel to (i).   
\end{proof}

Together with Propositions \ref{decomposition-X_K} and \ref{decomposition-X}, we obtain:  
\begin{corollary}\label{cor-level}
We have isomorphisms{\rm :} 
\begin{enumerate}
	\item $h^1(X_{N,K_N}) \simeq \bigoplus_{N'|N} \bigoplus_{(a',b') \in I_{N'}^\prim} 
	\varphi_{K_N/K_{N'}}^*X_{N'}^{a',b'}$ \ in $\sM_{K_N,E_N}$,
	\item $h^1(X_N) \simeq \bigoplus_{N'|N} \bigoplus_{[a',b']_k \in H_{N',k} 
	\backslash I_{N'}^\prim}  X_{N'}^{[a',b']_k}$ \ in $\sM_{k,E_{N,k}}$. 
\end{enumerate}
\end{corollary}

Next, the relation between $X_N^{a,b}$ and $X_{N}^{[a,b]_k}$ is as follows. 
\begin{proposition}\label{prop-field} \ 
\begin{enumerate}
\item For $(a,b) \in G_N$, we have an isomorphism in $\sM_{K_N,E_N}${\rm :} 
$$\varphi_{K_N/k}^* X_N^{[a,b]_k} \simeq \bigoplus_{(c,d) \in[a,b]_k} X_N^{c,d}.$$
\item For $(a,b) \in G_N^\prim$,  we have an isomorphism in $\sM_{k,E_N}${\rm :} 
$$\varphi_{K_N/k*}X_N^{a,b} \simeq X_N^{[a,b]_k}.$$ 
\end{enumerate}
\end{proposition}

\begin{proof}
(i) is clear from the definitions. 

We prove (ii). Recall that $\vphi_{K_N/k *}X_N^{a,b} = ({X_{N,K_N}},{p_N^{a,b}})$, 
viewed as a scheme and a correspondence over $k$. 
Consider the following diagram in $\sM_{k,E_N}$ with $\pi=\pi_N$ viewed as a $k$-morphism: 
\begin{equation*}\label{four-3}
\xymatrix{
h(X_N) \ar[d]_{p_N^{[a,b]_k}}\ar[r]^{p_N^{a,b}\circ\pi^*}  & h(X_{N,K_N}) \ar[r]^{\pi_*\circ p_N^{a,b}} \ar[d]_{p_N^{a,b}} & h(X_N) \ar[r]^{p_N^{a,b}\circ\pi^*} \ar[d]_{p_N^{[a,b]_k}} & h(X_{N,K_N}) \ar[d]_{p_N^{a,b}} \\
h(X_N)\ar[r]^{p_N^{a,b}\circ \pi^*}
& h(X_{N,K_N}) \ar[r]^{\pi_*\circ p_N^{a,b}} & h(X_N)  \ar[r]^{p_N^{a,b}\circ\pi^*}  & h(X_{N,K_N}) .
}\end{equation*}
It suffices to show the commutativity of the first square after applying $\circ \pi_*$. 
Then, using \eqref{corr-pull} and \eqref{corr-push}, we have  
\begin{align*}
p_N^{a,b} \circ \pi^*\circ p_N^{[a,b]_k} \circ \pi_* 
= p_N^{a,b} \circ (\pi\times\pi)^* p_N^{[a,b]_k}  
= p_N^{a,b} \circ \sum_{h \in H_{N,k}} h \circ p_N^{a,b} \circ \sum_{h \in H_{N,k}} h \\
= \sum_{h \in H_{N,k}} (h \circ p_N^{ha,hb}) \circ p_N^{a,b} \circ \sum_{h \in H_{N,k}} h 
= p_N^{a,b} \circ \sum_{h \in H_{N,k}} h  
=p_N^{a,b} \circ p_N^{a,b} \circ \pi^*\circ  \pi_* .
\end{align*}
The second square is again the ``transpose" of the first one.

Now, since $(a,b)$ is primitive, we have by definition
\begin{equation}\label{pi-p-pi}
\pi_* \circ p_N^{a,b} \circ\pi^* = (\pi \times \pi)_* p_N^{a,b} = p_N^{[a,b]_k},
\end{equation}
hence $\pi_* \circ p_N^{a,b}\circ\pi^* \circ p_N^{[a,b]_k} = p_N^{[a,b]_k}$. 
On the other hand, we have
$$p_N^{a,b}\circ\pi^* \circ \pi_* \circ p_N^{a,b}
= p_N^{a,b} \circ \sum_{h \in H_{N,k}} h \circ p_N^{a,b} 
= \sum_{h \in H_{N,k}} h p_N^{ha,hb} \circ p_N^{a,b} = p_N^{a,b}.$$
Therefore, 
$$p_N^{a,b}\circ \pi^* \colon  X_N^{[a,b]_k} \lra \vphi_{K_N/k *}X_N^{a,b}, \quad
\pi_*  \colon \vphi_{K_N/k *}X_N^{a,b} \lra X_N^{[a,b]_k}$$ 
are isomorphisms inverse to each other. 
\end{proof}

Combining Corollary \ref{cor-level} (ii) and Proposition \ref{prop-field} (ii) we obtain: 
\begin{corollary}
We have an isomorphism in $\sM_{k,E_N}${\rm :} 
$$h^1(X_N) \simeq \bigoplus_{N'|N} \bigoplus_{[a',b']_k \in H_{N',k}\backslash I_{N'}^\prim }
\vphi_{K_{N'}/k *} X_{N'}^{a',b'},$$
where $(a',b')$ is any representative of $[a',b']_k$. 
\end{corollary}

\section{$L$-functions of Fermat motives}
  
\subsection{$\ell$-adic realization of motives}

For a scheme $X$ and a prime number $\ell$ invertible on $X$, 
let $\mu_{\ell^n}$ be the \'etale sheaf of $\ell^n$-th roots of unity. 
For an integer $m$, we write $\Z/\ell^n\Z(m)  =\mu_{\ell^n}^{\ot m}$. 
The {\em $\ell$-adic \'etale cohomology group} is defined by  
$$
H_\et^i(X,\Ql(m)) = \Ql \ot_\Zl  \varprojlim_{n} H_\et^i(X,\Z/\ell^n\Z(m)).  
$$
Let $k$ be a field with $\mathrm{char}(k) \neq \ell$, and $\ol k$ be an algebraic closure of $k$. 
If $X \in \sV_k$, then
$$H_\ell^i(X)(m) := H^i_\et(X_{\ol k},\Ql(m)). $$
is a finitely generated $\Ql$-module on which 
 the absolute Galois group $\Gal(\ol k/k)$ acts continuously. 
If $X$ is a projective smooth curve over $k$, then 
$H^1_\ell(X)(1)$ is isomorphic to the $\ell$-adic Tate module of the Jacobian variety of $X$.

Let $X$, $Y \in \sV_k$ with $d = \dim X$, $d'=\dim Y$. 
For a correspondence  $f \in \Corr^r(X,Y) = \Q \ot_\Z \CH^{d +r}(X \times Y)$, 
let 
$$[f] \in H_\ell^{2(d+r)}(X \times Y)(d+r)$$ 
denote the cycle class of $f$. 
Consider the composition: 
\begin{multline*}
H_\ell^i(X)(m) \os{\mathrm{pr}_X^*}{\lra} H_\ell^i(X \times Y)(m) \\
\os{\cup [f]}{\lra} H_\ell^{i+2(d +r)}(X \times Y)(m+d+r)  \os{\mathrm{pr}_{Y*} }{\lra} H_\ell^{i+2r}(Y)(m+r),   
\end{multline*}
which we also denote by $f$. 
Here, $\cup$ is the cup product and 
the homomorphism $\mathrm{pr}_{Y*}$ is the dual of 
\begin{equation*}
H_\ell^{2d'-i-2r}(Y)(d'-m-r)  \os{\mathrm{pr}_Y^*}{\lra} 
H_\ell^{2d'-i-2r}(X \times Y)(d'-m-r)
\end{equation*}
via the Poincar\'e duality. 

For a motive $M=(X,p,m) \in \sM_{k,E}$, its $\ell$-adic cohomology is defined by
\begin{equation*}
H_\ell^i(M) = p (E \ot_\Q H_\ell^i(X)(m)).
\end{equation*}
Then, $H_\ell = \oplus_i H_\ell^i$ extends to the covariant {\em $\ell$-adic realization functor} 
\begin{equation*}
H_\ell \colon \sM_{k,E} \lra \Mod_{E_\ell[\Gal(\ol k/k)]}
\end{equation*}
to the category of modules over 
$$E_\ell :=E\ot_\Q \Ql$$ 
with Galois action. 
If $X \in \sV_k$ is a curve, then we have 
$H_\ell(h^i(X)) = H^i_\ell(X)$. 
Note that 
$$H_\ell(M(r))= H_\ell(M)(r) := H_\ell(M) \ot_\Ql \Ql(1)^{\ot r}$$
where $\Ql(1)$ is a one-dimensional $\Ql$-vector space on which $\Gal(\ol k/k)$ acts via the 
$\ell$-adic cyclotomic character. 

For any (resp. finite separable) extension $k'/k$ in $\ol k$,   
we have commutative diagrams of functors 
\begin{equation}\label{diagram-realization}
\begin{split}
\xymatrix{
\sM_{k,E} \ar[r]^(.4){H_\ell} \ar[d]_{\vphi_{k'/k}^*} &  \Mod_{\El[\Gal(\ol k/k)]} \ar[d]^{\Res_{k'/k}} \\ 
\sM_{k',E} \ar[r]^(.4){H_\ell} &   \Mod_{\El[\Gal(\ol k/k')]}
} 
\quad
\xymatrix{
\sM_{k',E} \ar[r]^(.4){H_\ell} \ar[d]_{\vphi_{k'/k,*}} & \Mod_{\El[\Gal(\ol k/k')]} \ar[d]^{\Ind_{k'/k}} \\
\sM_{k, E} \ar[r]^(.4){H_\ell}  &  \Mod_{\El[\Gal(\ol k/k)]} 
}
\end{split}\end{equation}
where $\Res_{k'/k}$ is the restriction of the Galois action, and 
$$\Ind_{k'/k}(V) =  \El[\Gal(\ol k/k)]  \ot_{\El[\Gal(\ol k/k')]} V$$
is the induced Galois module.

\subsection{$L$-functions of motives}

Let $k$ be a number field and $\sO_k$ be its integer ring. 
For a finite place $v$ of $k$, let $\sO_v$, $k_v$ be the completion of $\sO_k$, $k$, respectively, 
and $\F_v$ be the residue field; put $N(v)=\sharp \F_v$. 
Let 
$I_v \subset D_v \subset \Gal(\ol k/k)$
be the inertia and the decomposition subgroups at $v$, respectively, and 
$\Fr_v \in \Gal(\ol{\F_v}/\F_v) \simeq D_v/I_v$
be the geometric Frobenius of $\F_v$, i.e. the inverse of the 
the $N(v)$-th power Frobenius map.  

Let $E$ be a number field. For a prime number $\ell$, we have a natural decomposition 
\begin{equation*}
E_\ell = \prod_{\l | \ell} E_\l,
\end{equation*}
where $\l$ runs through the places of $E$ over $\ell$, and $E_\l$ is the completion.  
For an $E_\ell$-module $V$, let 
\begin{equation*}
V = \bigoplus_{\l |\ell} V_\l, \quad V_\l := E_\l \ot_{E_\ell} V
\end{equation*}
be the corresponding decomposition into $E_\l$-modules.  

Let $M=(X,p,m) \in \sM_{k,E}$ be a motive. For each finite place $v$ of $k$, 
choose $\ell \neq \mathrm{char}(\F_v)$ and a place $\l | \ell$ of $E$.  
Then, the zeta polynomial of $M$ at $v$ is defined by 
\begin{equation*}
P_v(M,T) = \det\left(1-\Fr_vT; H_\l(M)^{I_v}\right)  \ \in E_\l[T]
\end{equation*}
where we write $H_\l(M)= H_\ell(M)_\l$.  

\begin{conjecture}[$\ell$-independence]\label{l-independence}
Let $M \in \sM_{k,E}$. For any finite place $v$ of $k$, 
$P_v(M,T) \in \sO_E[T]$, and is independent of the choice of $\ell$ and $\l$.
\end{conjecture}

Assume the conjecture for $M$. 
For each embedding $\s \colon E \hookrightarrow \C$, the {\em $L$-factor} at $v$ is  
defined by
\begin{equation*}
L_v(\s,M,s) = \s P_v(M,N(v)^{-s})^{-1}, 
\end{equation*}
and the {\em $L$-function} is defined by
\begin{equation*}
L(\s,M,s) = \prod_{v} L_v(\s,M,s). 
\end{equation*}
We denote by $L(M,s)$ the system $(L(\s,M,s))_\s$, which may be viewed as an $E_\C$-valued function,   
where
\begin{equation*}
E_\C := E \ot_\Q \C = \prod_{\s\colon E \hookrightarrow \C} \C.
\end{equation*}
Note the relation
\begin{equation*}
L(M(r),s)=L(M,s+r). 
\end{equation*}
We also define the $L$-function $L(h^i(M),s)$ of the (conjectural) $i$-th motive 
to be the $L$-function associated to its ``realization" $H^i_\ell(M)$. 

We shall use the following proposition, which follows from \eqref{diagram-realization}.   
\begin{proposition}\label{induce-L}
Let $k'/k$ be a finite extension of number fields and 
suppose that  a motive $M \in \sM_{k',E}$ satisfies Conjecture \ref{l-independence}. 
Then,  
$\vphi_{k'/k *}M$ satisfies Conjecture \ref{l-independence}
and we have $$L(\vphi_{k'/k *}M, s) = L(M,s).$$
\end{proposition}

If $X \in \sV_k$, it has good reduction at almost all $v$, 
i.e. there exists a proper smooth model over $\sO_v$ with generic fiber $X \times_k k_v$; 
denote the special fiber by $X_{\F_v}$.  
We say that 
$M=(X,p,m)$ has good reduction at $v$ 
if $X$ and all the components of a cycle representing $p$ have good reductions.   
For such $M$ and $v$, $H_\ell(M)$ is unramified at $v$, i.e. the action of $I_v$ is trivial, 
and there exists an isomorphism
$$H_\ell(M) \simeq H_\ell(M_{\F_v})$$
of $\Gal(\ol{\F_v}/\F_v)$-modules (see \cite{sga4demi}). 
Therefore, 
\begin{equation}\label{zeta-zeta}
P_v(M,T) = P(M_{\F_v},T):= \det(1-\Fr_{\F_v}T; H_\l(M_{\F_v}))
\end{equation}
and Conjecture \ref{l-independence} holds \cite{deligne-weil1}. 

Suppose that $H_\ell(M)=H^i_\ell(M)$. 
Then, by the Weil conjecture proved by Deligne \cite{deligne-weil1}, for good $v$, 
$P(M_{\F_v},T)$  is of pure weight $w=i-2m$, i.e. any complex conjugate of a reciprocal root has complex absolute value $N(v)^{w/2}$. 
Therefore, $L(M,s)$ except for the bad factors 
converges absolutely for $\Re(s)> w/2+1$ and has no zero nor pole in the region. 
Further, the weight-monodromy conjecture \cite{deligne-weil2} implies that the bad factors either
have no zero nor pole in the same region (cf. \cite{schneider}).  

\subsection{Jacobi sums} 

Let $K$ be a finite field of characteristic $p$ with $q$ elements which contains 
the $N$-th roots of unity, i.e. $N \mid q-1$.  
Fix an isomorphism 
\begin{equation}\label{mumu} 
\mu_N(K) \os{\sim}{\lra} \mu_N(E_N).
\end{equation}
Composing the $\frac{q-1}{N}$-th power map 
$K^* \ra \mu_N(K)$
with \eqref{mumu}, we obtain a character 
\begin{equation*}
\chi_N \colon K^* \lra \mu_N(E_N)
\end{equation*}
of exact order $N$. 
We extend $\chi_N^a$ to the whole $K$ by setting $\chi_N^a(0)=0$. 
For $(a,b) \in G_N$, the {\em Jacobi sum} is defined by 
\begin{equation*}
j(\chi_N^a,\chi_N^b) = - \sum_{x,y \in K, x+y=1} \chi_N^a(x) \chi_N^b(y) \ \in E_N.
\end{equation*}

Fix an algebraic closure $\ol K$ of $K$ and let $K_n/K$ be the subextension of degree $n$. 
The character $\chi_{N,n} \colon K_n^* \ra E_N^*$ is defined as above via 
$\mu_N(K_n) = \mu_N(K) \simeq \mu_N(E_N)$, or equivalently, $\chi_{N,n} =\chi_N \circ N_{K_n/K}$. 
Then, for any $(a,b) \in I_N$, we have the Davenport-Hasse relation
\begin{equation}\label{davenport-hasse}
j(\chi_{N,n}^a,\chi_{N,n}^b) = j(\chi_N^a,\chi_N^b)^n.
\end{equation}

\subsection{Zeta polynomials of Fermat motives}

Let $k$ be a finite field of characteristic $p \nmid N$, $\ol k$ be an algebraic closure of $k$ and 
put $K=K_N=k(\mu_N)$.   
By choosing $\z_N \in K_N$ and $\xi_N \in E_N$, the motives 
$X_N^{a.b} \in \sM_{K_N,E_N}$ and $X_N^{[a,b]_k} \in \sM_{k,E_{N,k}}$
are defined. Note that $H_{N,k} \subset (\Z/N\Z)^*$ is the cyclic subgroup generated by $\sharp k$. 
Fix the isomorphism \eqref{mumu} by assigning $\z_N$ to $\xi_N$. 
Then the character $\chi_N$ and the Jacobi sums are defined. 

We calculate the zeta polynomials of $X_N^{a,b}$ and $X_N^{[a,b]_k}$. 
We only consider the case $(a,b) \in I_N$, whereas the other cases are  
obvious by Propositions \ref{decomposition-X_K} and \ref{decomposition-X}.  
The key is the Grothendieck's fixed point formula (cf. \cite{sga4demi}): 
for an endomorphism $F$ of $X$ over $k$, we have 
\begin{equation}\label{fixed-point-formula}
\sharp X(\ol k)^F = \sum_i (-1)^i \Tr(F; H^i_\ell(X)).
\end{equation}

\begin{theorem}\label{fermat-polynomial}\ 
\begin{enumerate}
\item If $(a,b) \in I_N$, then $P(X_N^{a,b}, T) = 1- j(\chi_N^a,\chi_N^b) T$. 
\item If $(a,b) \in I^\prim_N$, then $P(X_N^{[a,b]_k}, T) = 1- j(\chi_N^a,\chi_N^b) T^{\sharp[a,b]_k}$. 
\end{enumerate}
In particular, these do not depend on the choice of $\l$ and belong to $\sO_{E_{N,k}}[T]$. 
\end{theorem}

\begin{proof}
(i) First, by taking the logarithm, we have
\allowdisplaybreaks
\begin{align*}
&\log \det\bigl(1-\Fr_{K} T; H_\l^1(X_N^{a,b})\bigr)  \\
&=\log  \prod_{i=0,1,2} \det\bigl(1-\Fr_{K} T; H_\l^i(X_N^{a,b})\bigr)^{(-1)^{i+1}}  \\
&= \log  \prod_{i=0,1,2} \det\bigl(1-\Fr_{K} p_N^{a,b} T; H_\l^i(X_{N,K})\bigr)^{(-1)^{i+1}}\\
&=  \sum_{i=0,1,2} (-1)^{i} \sum_{n \geq 1} \frac{1}{n} \Tr\bigl((\Fr_K p_N^{a,b})^n; H_\l^i(X_{N,K})\bigr) T^n\\
&=  \sum_{i=0,1,2} (-1)^{i} \sum_{n \geq 1} \frac{1}{n} \Tr\bigl(\Fr_K^n p_N^{a,b}; H_\l^i(X_{N,K})\bigr) T^n\\
&=  \sum_{n \geq 1} \frac{1}{n}  \left(\frac{1}{N^2} \sum_{g \in G_N} \theta_N^{a,b}(g^{-1})  \sum_{i=0,1,2} (-1)^{i} \Tr\bigl(\Fr_K^n g; H_\l^i(X_{N,K})\bigr) \right) T^n . \\
\end{align*}
Then by \eqref{fixed-point-formula},  
the alternating sum of the traces equals
$$\vL(\Fr_K^n g) := \sharp\bigl\{P \in X_{N}(\ol k) \bigm| \Fr_K^ng(P)=P\bigr\}.$$ 
This is devided as $\vL(\Fr_K^n g)=\vL_0(\Fr_K^n g)+\vL_1(\Fr_K^n g)$ 
with
\begin{align*}
&\vL_0(\Fr_K^n g) := \sharp\bigl\{P \in U_{N}(\ol k) \bigm| \Fr_K^ng(P)=P\bigr\}, \\
&\vL_1(\Fr_K^n g) :=  \sharp\bigl\{P \in Z_{N}(\ol k) \bigm| \Fr_K^ng(P)=P\bigr\},  
\end{align*}
where $U_N$, $Z_N \subset X_N$ are the subschemes defined in \S 2.4. 

Let $q = \sharp K$. If $g=g_N^{-r,-s}$, then we have 
\begin{align*}
& \vL_0(\Fr_K^n g) \\
&= \sharp \bigl\{(x,y) \in \ol k^2 \bigm| x^N+y^N=1, x^{q^n} = \z_N^r x, y^{q^n}=\z_N^s y \bigr\} \\
&= \sum_{u+v=1}  \sharp\bigl\{x \in \ol k \bigm| x^N=u, x^{q^n}=\z_N^r x\bigr\} \sharp\bigl\{y \in \ol k \bigm| y^N=v, y^{q^n}=\z_N^s y\bigr\}. 
\end{align*}
Here, the sum is taken over $u$, $v \in K_n$, the extension of $K$ of degree $n$, 
since 
$u^{q^n}=x^{Nq^n}=(\z_N^rx)^N=x^N=u$. 
If $u \neq 0$, then  $x^N=u$ and $x^{q^n}=\z_N^r x$ imply that 
$u^{\frac{q^n-1}{N}} =x^{q^n-1}=\z_N^r$. Conversely, if $u^{\frac{q^n-1}{N}} =\z_N^r$, 
then any solution of $x^N=u$ satisfies $x^{q^n}=\z_N^r x$. 
Therefore, we have 
\begin{equation*}
\sharp\bigl\{x \in \ol k \bigm| x^N=u, x^{q^n}=\z_N^r x\bigr\} = \begin{cases}
N \cdot  \delta \bigl(u^{\frac{q^n-1}{N}}=\z_N^r\bigr) & \text{if $u \neq 0$} \\
1 & \text{if $u=0$} 
\end{cases} 
\end{equation*}
where $\delta (\mathrm{P})$ is $1$ (resp. $0$) if the statement $\mathrm{P}$ is true (resp. false). 
It follows that 
\begin{multline*}
\vL_0(\Fr_K^n g) 
= N^2 \sum_{u,v \in K_n^*, u+v=1} \delta \bigl(u^{\frac{q^n-1}{N}}=\z_N^r\bigr) \delta \bigl(v^{\frac{q^n-1}{N}}=\z_N^s\bigr) \\ 
+ N \delta(r=0) + N \delta(s=0).  
\end{multline*}
On the other hand, 
\begin{align*}
& \vL_1(\Fr_K^n g)\\
& =\sharp \bigl\{(x_0 : y_0) \in \P^1(\ol k) \bigm| 
x_0^N+y_0^N=0, (x_0^{q^n}:y_0^{q^n}) = (\z_N^rx_0 : \z_N^s y_0) \bigr\}\\
&= \sharp \bigl\{w \in \ol k^* \bigm| w^N=-1, w^{q^n}=\z_N^{r-s}w \bigr\} \\
&= N \delta\bigl((-1)^{\frac{q^n-1}{N}}=\z_N^{r-s}\bigr).
\end{align*}

Now, by the definition of $\chi_{N,n}$,  we have 
$\chi_{N,n}(u)=\xi_N^r$ for $r$ such that $u^{\frac{q^n-1}{N}}=\z_N^r$. 
Therefore, we have 
\begin{align*}
&\sum_{r,s} \theta_N^{a,b}(g_N^{r,s}) \sum_{u,v \in K_n^*, u+v=1} 
\delta \bigl(u^{\frac{q^n-1}{N}}=\z_N^r\bigr) \delta \bigl(v^{\frac{q^n-1}{N}}=\z_N^s\bigr)  \\
&=\sum_{u,v \in K_n^*, u+v=1} \chi^a_{N,n}(u)\chi^b_{N,n}(v)
= -j(\chi_{N,n}^a,\chi_{N,n}^b)=-j(\chi_N^a,\chi_N^b)^n 
\end{align*}
by \eqref{davenport-hasse}. On the other hand, since $(a,b) \in I_N$, 
\begin{multline*}
\sum_{r,s} \theta_N^{a,b}(g_N^{r,s}) \left( \delta(r=0)+\delta(s=0)+
\delta\bigl((-1)^{\frac{q^n-1}{N}}=\z_N^{r-s}\bigr)\right) \\
= \sum_s \xi_N^{bs} + \sum_r \xi_N^{ar} + (-1)^{\frac{q^n-1}{N}a} \sum_s \xi_N^{(a+b)s} =0.
\end{multline*}
We obtained  
$$\frac{1}{N^2} \sum_{g \in G_N} \theta_N^{a,b}(g^{-1})\vL(\Fr_K^n g) 
=-j(\chi_N^a,\chi_N^b)^n $$
and hence
$$P(X_N^{a,b},T)= \exp  \sum_{n \geq 1} \frac{-j(\chi_N^a,\chi_N^b)^n}{n} T^n 
= 1-j(\chi_N^a,\chi_N^b) T. $$

The statement (ii) follows from (i), Proposition \ref{prop-field} (ii) and the general fact 
 $$P(\pi_{k'/k*}M',T)=P(M',T^{[k':k]})$$
for finite fields $k'/k$ and $M' \in \sM_{k',E}$. 
The final assertion follows since $j(\chi_N^{ap},\chi_N^{bp}) = j(\chi_N^a,\chi_N^b)$. 
\end{proof}

\begin{remark}
For $(a,b) \in I_N$ not necessarily primitive, let $N=N'd$, $a=a'd$, $b=b'd$, with $(a',b') \in I^\prim_{N'}$. Then we have
$$
P(X_N^{[a,b]_k},T)=P(X_{N'}^{[a',b']_k},T) =1-j(\chi_{N'}^{a'},\chi_{N'}^{b'})T^{\sharp[a,b]_k}
$$
(recall that $\sharp[a,b]_k=\sharp[a',b']_k$). 
\end{remark}

\begin{corollary}
If $(a,b) \in I_N$, then 
$H_\ell(X_N^{a,b})$ is a free $E_{N,\ell}$-module of rank $1$ and 
$H_\ell(X_N^{[a,b]_k})$ is a free $(E_{N,k})_\ell$-module of rank $\sharp[a,b]_k$.  
\end{corollary}

\begin{remark}
This also follows from the corresponding result for the singular cohomology (cf. \cite{otsubo})
and  the comparison theorem between \'etale and singular cohomology (Artin's theorem). 
\end{remark}

\subsection{$L$-functions of Fermat motives}

Now, let $k$ be a number field and $K_N =k(\mu_N) \subset \ol k$ as before.   
By choosing $\z_N \in \mu_N(K_N)$ and $\xi_N \in \mu_N(E_N)$, 
the motives $X_N^{a,b} \in \sM_{K_N,E_N}$ and $X_N^{[a,b]_k} \in \sM_{k,E_{N,k}}$  
are defined.   

By Proposition \ref{decomposition-X_K}, the cases $(a,b) \not\in I_N$ are easy: 
$$L(X_N^{0,0},s) = \z_{K_N}(s)\z_{K_N}(s-1), \quad 
L(X_N^{[0,0]_k},s) = \z_{k}(s)\z_{k}(s-1), $$
and $L(X_N^{a,b},s) =L(X_N^{[a,b]_k},s) =1$ if only one of $a$, $b$, $a+b$ is $0$. 
For $(a,b) \in I_N$, we are reduced to the primitive case over $K_N$: if $(a,b)=(a'd,b'd)$ with $N=N'd$, $(a',b') \in I_{N'}^\prim$, 
then we have
\begin{equation*}
L(X_N^{a,b},s) = L(X_{N',K_N}^{a',b'},s), \quad
L(X_N^{[a,b]_k},s)  = L(X_{N'}^{[a',b']_k},s) = L(X_{N'}^{a',b'},s)
\end{equation*}
by Propositions \ref{prop-level}, \ref{prop-field} and \ref{induce-L}.    

If $v \nmid N$ is a place of $K_N$, we have a canonical isomorphism 
$$\mu_N(K_N) \os{\simeq}{\lra} \mu_N(\F_v).$$
By abuse of notation, we also denote the image of $\z_N$ by the same letter. 
Then, using $\z_N$ and $\xi_N$, the motive $X_{N,\F_v}^{a,b} \in \sM_{\F_v,E_N}$ is defined.  
The  character 
$$\chi_{N,v}\colon \F_v^* \lra E_N^*$$
and the Jacobi sum 
\begin{equation*}
j_{N}^{a,b}(v) := j(\chi_{N,v}^a,\chi_{N,v}^b) 
\end{equation*}
are defined as in \S 3.3. 
At $v \nmid N$, the motive $X_N^{a,b}$ has good reduction: 
$(X_N^{a,b})_{\F_v}=X_{N,\F_v}^{a,b}$. 
By \eqref{zeta-zeta} and Theorem \ref{fermat-polynomial}, we have: 

\begin{proposition} Let $(a,b) \in I_N$ and $v$ be a finite place of $K_N$ not dividing $N$. Then we have 
$$P_v(X_N^{a,b},T)= 1-j_N^{a,b}(v) T.$$
\end{proposition}

Now we determine the bad $L$-factors. 
\begin{proposition}\label{bad-realization}
Let $(a,b) \in I_N^\prim$ and $v$ be a place of $K_N$ dividing $N$. 
Then we have 
$H_\ell(X_{N}^{a,b})^{I_v}=0$  for any $\ell\neq \mathrm{char}(\F_v)$. 
In particular, $P_v(X_N^{a,b},T)=1$. 
\end{proposition}

\begin{proof}
Let $X_{N,\F_v}$ denote the special fiber at $v$ of the model defined by the same equation as $X_N$. 
By a similar argument as in \cite{deligne-weil2} (3.6), we have a canonical surjection 
\begin{equation*}
\xymatrix{ H_\et^1(X_{N,\ol \F_v}, \Ql) \ar@{->>}[r] & H_\et^1(X_{N,\ol k}, \Ql)^{I_v}
}
\end{equation*}
compatible with the action of $\Fr_v$. 
Let $p=\mathrm{char}(\F_v)$, $N=p^eN'$ with $e>0, p \nmid N'$ 
and consider the commutative diagram 
\begin{equation*}
\xymatrix{
H_\et^1(X_{N',\ol \F_v}, \Ql) \ar[r]^\simeq \ar[d]_{\pi_{N/N',\ol \F_v}^*}
& H_\et^1(X_{N',\ol k}, \Ql) \ar@{^{(}->}[d]_{\pi_{N/N',\ol k}^*}
\phantom{^{I_v}. }\\
H_\et^1(X_{N,\ol \F_v}, \Ql) \ar@{->>}[r]  & H_\et^1(X_{N,\ol k}, \Ql)^{I_v}. 
}
\end{equation*}
Since $X_{N',K_N}$ has good reduction at $v$, the upper horizontal map is an isomorphism. The right vertical map is injective by the norm argument: $\pi_{N/N *} \circ \pi_{N/N'}^* = p^{2e}$. 
Consider a morphism
$$f \colon X_{N',\F_v} \lra X_{N,\F_v}; \quad (x_0:y_0:z_0) \longmapsto (x_0:y_0:z_0),$$
which identifies $X_{N',\F_v}$ with the reduced scheme associated to $X_{N,\F_v}$. 
Since $f$ is defined by a nilpotent ideal, it induces an isomorphism on cohomology. 
The composite $f\circ \pi_{N/N',\ol \F_v}$ coincides with the base change of $(F_{(p)})^e$ where 
$F_{(p)}$ is the Frobenius endomorphism of $X_{N,\F_p}$ \cite{sga4demi}, Rapport, \S 1. 
The action of $F_{(p)}$ on cohomology coincides with that of $\Fr_{\F_p}$, hence is an isomorphism.  Therefore, $\pi_{N/N',\ol \F_v}^*$ is surjective and all the maps in the diagram are isomorphic. By Proposition \ref{prop-level}, (i), we obtain
$$H_\ell^1(X_{N})^{I_v} = \bigoplus_{(a',b')\in I_{N'}} H_\ell(X_N^{a'p^e,b'p^e}),$$
which finishes the proof. 
\end{proof}

We have proved: 
\begin{theorem}\label{fermat-L} \ 
\begin{enumerate}
\item For any $(a,b) \in G_N$, Conjecture \ref{l-independence} is true for $X_N^{a,b}$ and $X_N^{[a,b]_k}$.  
\item
For $(a,b) \in I_{N}^\prim$ and an embedding $\s\colon E_N \hookrightarrow \C$, we have
\begin{align*}
L(\s,X_N^{a,b},s)=
L(\s_{|E_{N,k}},X_N^{[a,b]_k},s)=
 \prod_{v \nmid N} \Bigl(1-\s\bigl( j_{N}^{a,b}(v)\bigr) N(v)^{-s}\Bigr)^{-1}. 
\end{align*}
\end{enumerate}
\end{theorem}

\begin{remark}\label{sigma-independence}
It follows that $L(X_N^{a,b},s)$ depens only on the class $[a,b]_k$ and 
is an $(E_{N,k})_\C$-valued function.  
In particular, if $H_{N,k} = H_N$ (e.g. $k=\Q$), then it is $\C$-valued, 
i.e.  $L(\s,X_N^{a,b},s)$ does not depend on $\s$. 
\end{remark}

\subsection{Functional equation}

Let $k$ be a number field and for an infinite place $v$ of $k$, let $k_v$ be its completion. 
Then, for $X \in \sV_k$,  we have 
$$X_\R := X \times_\Q \R = \bigsqcup_{v| \infty} X_v, \quad X_v :=X \times_k k_v$$  
regarded as $\R$-schemes.  
More generally, for a motive $M \in \sM_{k,E}$, we have 
$$M_\R:=\vphi_{\R/\Q}^*\vphi_{k/\Q *}M = \bigoplus_{v | \infty} M_v, \quad M_v:=\vphi_{k_v/\R *}\vphi_{k_v/k}^*M$$
in $\sM_{\R,E}$.  

For a motive $M=(X,p,m) \in \sM_{\R,E}$ over $\R$, the $i$-th singular cohomology is defined by
$$H^i(M(\C),\Q) := p\left(E \ot_\Q H^i(X(\C),\Q(m))\right)$$ 
(we use such a notation although  $M(\C)$ itself is not defined),   
where 
$$\Q(m):=(2\pi i)^m \Q.$$
It is a Hodge structure of pure weight $w=i-2m$, that is, 
there is a bigrading  
$$H^i(M(\C),\Q) \ot_\Q \C \simeq \bigoplus_{p+q=w}H^{p,q}(M)$$
as an $E_\C$-module such that the complex conjugation $c_\infty$ (on the coefficients) 
exchanges $H^{p,q}$ and $H^{q,p}$.  
Since $X$ is defined over $\R$, the complex conjugation $F_\infty$ called the {\em infinite Frobenius} acts on $X(\C)$ and hence on the cohomology, which also exchanges $H^{p,q}$ and $H^{q,p}$. 
The Hodge numbers are defined by
$$h^{p,q}(M)=\rank_{E_\C} H^{p,q}(M), \quad h_\pm^{p,p}(M) = \rank_{E_\C} H_\pm^{p,p}(M),$$
where  
$H_\pm^{p,p}(M) := H^{p,p}(M)^{F_\infty = \pm (-1)^p}$. 
Using the standard notations 
\begin{equation*}
\vG_\R(s) :=\pi^{-s/2}\vG(s/2),\quad \vG_\C(s) := \vG_\R(s)\vG_\R(s+1)=2(2\pi)^{-s}\vG(s),
\end{equation*}
we put   
\begin{equation*}
\vG(M,s) = \prod_{p<q} \vG_\C(s-p)^{h^{p,q}(M)} 
\prod_{p}  \vG_\R(s-p)^{h_+^{p,p}(M)} \vG_\R(s+1-p)^{h_-^{p,p}(M)} . 
\end{equation*}

Now, for a motive $M \in \sM_{k,E}$ over a number field, assume Conjecture \ref{l-independence}
and define the  {\em completed $L$-function} by
$$
\vL(M,s)=L(M,s)\vG(M_\R,s). 
$$
By the Poincar\'e duality
$H_\ell^i(M)^\vee \simeq H_\ell^{2d-i}(M^\vee)$, 
Conjecture \ref{l-independence} also holds for $M^\vee$. 

\begin{conjecture}[Hasse-Weil]\label{hasse-weil} 
$L(M,s)$ is continued to a meromorphic function on the whole complex plane  
and satisfies a functional equation 
$$\vL(M,s) = \varepsilon(M,s) \vL(M^\vee,1-s)$$
where $\varepsilon(M,s)$ is the product of a constant and an exponential function 
(see \cite{deligne-valeurs}, \cite{serre-facteurs}). 
\end{conjecture}

\begin{remark}
If $M=h^i(X)$, then by the hard Lefschetz theorem
$H_\ell^{2d-i}(X)(d) \simeq H_\ell^i(X)(i)$, 
the functional equation is also written as
$$\vL(h^i(X),s) = \varepsilon(h^i(X),s) \vL(h^i(X), i+1-s).$$ 
\end{remark}

Now, consider the Fermat motive $M=X_N^{a,b} \in \sM_{K_N,E_N}$, $(a,b) \in I_N$. 
Recall that $M^\vee = X_N^{-a,-b}(1)$ (Remark \ref{fermat-dual}). 
Weil \cite{weil-jacobi} proved that $j_N^{a,b}$ is a Hecke character of conductor dividing $N^2$, 
hence by Theorem \ref{fermat-L}, $L(X_N^{a,b},s)$ satisfies Conjecture \ref{hasse-weil}.  
As we shall see later (Remark \ref{remark-h-number}), for each infinite place $v$ of $K_N$, 
we have
\begin{equation}\label{h-number}
h^{0,1}(M_v)=h^{1,0}(M_v)=1
\end{equation} 
and the others are $0$. 
Therefore, we have
\begin{equation*}
\vL(X_N^{a,b},s)= L(X_N^{a,b},s)\vG_\C(s)^{r_2}
\end{equation*}
where $r_2 = [K_N:\Q]/2$ is the number of the complex places of $K_N$, so we obtain: 
\begin{corollary}\label{functional-equation}
Let $(a,b) \in I_N$. 
\begin{enumerate}
\item $L(X_N^{a,b},s)$ is analytically continued to an entire function on the whole complex plane. 
\item $\vL(X_N^{a,b},s) = \a\b^{s} \cdot \vL(X_N^{-a,-b}, 2-s)$ with some $\a, \b \in (E_N)_\C^*$. 
\item $L(X_N^{a,b},s)$ has zero of order $r_2$ at each non-positive integer. 
\end{enumerate}
\end{corollary}

\begin{remark}\label{epsilon}
If $N=p$ is a prime number and $K_p=\Q(\mu_p)$, then 
the $\varepsilon$-factor is classically known by Hasse \cite{hasse} (cf. \cite{gross-rohrlich}):   
$$\varepsilon(X_p^{a,b},s)=\pm (p^{p-2+f})^{s-1}$$
with $f=1$ or $2$, easily calculated from $(a,b)$. 
The sign (root number) is determined by Gross-Rohrlich \cite{gross-rohrlich}. 
\end{remark}

\subsection{Artin $L$-functions}

In \cite{weil-jacobi}, Weil interpreted the Jacobi-sum Hecke $L$-function as an Artin $L$-function. 
Our motivic $L$-function is regarded as a rephrasing of it.  
Though not necessary in the sequel, we explain the relation between them.  
Although the representation $D_{a,b}$ of Weil (loc. cit.) is not written explicitly, our $\rho_N^{[a,b]_k}$ below should correspond to it.  

Let $\sX$ be a scheme of finite type over $\Spec \Z$ and $|\sX|$ be the set of its closed points. 
For $x \in |\sX|$, let $\kappa(x)$ be its (finite) residue field and put $N(x) = \sharp \kappa(x)$. 
The {\em Hasse zeta function} of $\sX$ is defined by 
\begin{equation*}
\z(\sX,s) = \prod_{x \in |\sX|} (1- N(x)^{-s})^{-1}.
\end{equation*}
If $\sX$ is of Krull dimension $d$, $\z(\sX,s)$ converges absolutely for $\Re(s)>d$ 
and defines a holomorphic function in the region.   

Let $\sX \ra \sY$ be a finite flat covering of schemes of finite type over $\Z$ which is 
generically \'etale and Galois with Galois group $G$. 
Let $\rho$ be a complex representation of $G$ and $\chi$ be its character.   
The {\em Artin $L$-function} $L(\sX/\sY,\rho,s)$ is defined by
\begin{equation*}
\log L(\sX/\sY,\rho,s) = \sum_{y \in |Y|} \sum_{n=1}^\infty \frac{\chi(y^n) N(y)^{-s}}{n}
\end{equation*}
(see \cite{serre-zeta}).
If $\rho$ is the unit representation (resp. the regular representation), then it reduces to 
$\z(\sY,s)$ (resp. $\z(\sX,s)$). 

Now, let $\sX_{N,k}$ be the Fermat scheme of degree $N$ over $\sO_k$ 
defined by the same equation \eqref{equation-fermat},  
and consider the diagram similar to \eqref{diagram-1}. 
By the basic functorialities \cite{serre-zeta}, we have    
\begin{align*}
\z(\sX_{N,k},s) 
& = L(\sX_{N,K_N}/\sX_{N,k}, 1_{H_{N,k}},s)\\
& =L(\sX_{N,K_N}/\sX_{1,k}, \Ind_{H_{N,k}}^{\vG_{N,k}} 1_{H_{N,k}}, s).  
\end{align*}
It is not difficult to determine the irreducible decomposition  of $\Ind_{H_{N,k}}^{\vG_{N,k}} 1_{H_{N,k}}$. 
For $(a,b) \in G_N$,  there is a unique $(a',b')\in G_{N'}^\prim$ with $N=N'd$ such that $(a,b)=(a'd,b'd)$. 
Put 
\begin{align*}
\rho_N^{[a,b]_k} = \Res_{\vG_{N',k}}^{\vG_{N,k}} \Ind_{G_{N'}}^{\vG_{N',k}} \s\theta_{N'}^{a',b'}, 
\end{align*}
where 
$$\s\theta_N^{a,b} \colon G_N \os{\theta_N^{a,b}}{\lra} E_N^*  \os{\s}{\hookrightarrow} \C^*$$ 
is the composition. 
Then one shows that $\rho_N^{[a,b]_k}$ is irreducible and  
\begin{equation*}
\Ind_{H_{N,k}}^{\vG_{N,k}} 1_{H_{N,k}} = \bigoplus_{[a,b]_k \in H_{N,k}\backslash G_N} \rho_N^{[a,b]_k}. 
\end{equation*}
Therefore, we obtain  
\begin{equation*}
\z(\sX_{N,k},s)=\prod_{[a,b]_k \in H_{N,k}\backslash G_N} L(\sX_{N,K_N}/\sX_{1,k}, \rho_N^{[a,b]_k},s). 
\end{equation*}
Further, we have
\begin{align*}
L(\sX_{N,K_N}/\sX_{1,k},\rho_N^{[a,b]_k},s) 
&= L(\sX_{N',K_{N'}}/\sX_{1,k}, \Ind_{G_{N'}}^{\vG_{N',k}} \s\theta_{N'}^{a',b'},s)\\
&=L(\sX_{N',K_{N'}}/\sX_{1,K_{N'}}, \s\theta_{N'}^{a',b'},s). 
\end{align*}
\begin{proposition}
For $(a,b) \in I_{N}^\prim$,  we have
\begin{equation*}
L(\sX_{N,K_{N}}/\sX_{1,K_{N}}, \s\theta_{N}^{a,b},s) = L(\s,X_{N}^{a,b},s)^{-1}. 
\end{equation*}
\end{proposition}
\begin{proof}
We prove it fiberwise; 
let $v$ be a finite place of $K_N$.
Then, we have
$$\log L(\sX_{N,\F_v}/\sX_{1,\F_v},\s\theta_N^{a,b},s) = \sum_{n=1}^\infty \frac{\nu_n T^n}{n}, \quad
T=N(v)^{-s}$$
where
$$\nu_n = \frac{1}{N^2} \sum_{g \in G_N} \s\theta_N^{a,b}(g^{-1}) \vL(\Fr_{\F_v}^ng)$$
(see \cite{serre-zeta}). 
If $v \nmid N$, it equals $\log P_v(X_N^{a,b},T)$ by the proof of Theorem \ref{fermat-polynomial}. 
If $v |N$, let $p=\mathrm{char}(\F_v)$ and $N'=N/p$. 
Then, the action of $g \in G_N$ on $\sX_{N,\F_v}$ depends only on the image of $g$ in $G_{N'}$. 
Since $(a,b)$ is primitive, we have 
$\sum_{g \in G_{N/N'}} \theta_N^{a.b}(g^{-1})=0$, 
hence $\nu_n=0$, and the proof finishes by Proposition \ref{bad-realization}. 
\end{proof}

If $N'=1$, i.e. $(a,b)=(0,0)$, then we are reduced to 
$$\z(\sX_{1,k},s) = \z(\P^1_{\sO_k},s) =  \z_{k}(s)\z_{k}(s-1).$$
If only one of $a$, $b$, $a+b$ is $0$, one proves easily that $L(\sX_{N,K_{N}}/\sX_{1,K_{N}}, \s\theta_{N}^{a,b},s) =1$. 
Summarizing, we obtain:
\begin{proposition}
$$\z(\sX_{N,k},s)= \z_{k}(s)\z_{k}(s-1) \prod_{N'|N} \prod_{[a',b']_k \in H_{N',k}\backslash I_{N'}^\prim} 
L(\s,X_{N'}^{a',b'},s)^{-1}.$$
\end{proposition}

\section{Regulators of Fermat motives}

\subsection{Motivic cohomology}

We briefly recall the definition of motivic cohomology of motives and its integral part.  
For more details, see \cite{nekovar}, \cite{schneider}, \cite{scholl-2}.

For a noetherian scheme $X$, let $K_i(X)$ (resp. $K'_i(X)$) be the algebraic $K$-group 
of vector bundles (resp. coherent sheaves) \cite{quillen}. 
If $X$ is regular, the natural map $K_i(X)  \ra K_i'(X)$ is an isomorphism. 
For a quasi-projective variety $X$ over a field, we define its {\em motivic cohomology group} by 
$$H^n_\sM(X,\Q(r)) = K_{2r-n}(X)_\Q^{(r)}, $$
the Adams eigenspace of weight $r$ \cite{soule}.  

Recall the Grothendieck Riemann-Roch theorem
$\Q \ot_\Z \CH^r(X) = K_0(X)_\Q^{(r)}$. 
For $X$, $Y \in \sV_k$ and $f \in \Corr^d(X,Y) = K_0(X\times Y)_\Q^{(\dim X + d)}$ 
(for $X$  irreducible), 
the composition
\begin{multline*}
K_i(X)_\Q \os{\pr_X^*}{\lra} K_i(X\times Y)_\Q \os{\cup f}{\lra} K_i(X \times Y)_\Q =K'_i(X \times Y)_\Q \\
\os{\pr_{Y *}}{\lra} K'_i(Y)_\Q=K_i(Y)_\Q
\end{multline*}
induces a homomorphism
$H_\sM^n(X,\Q(r)) \ra H_\sM^{n+2d}(Y,\Q(r+d))$
by the Riemann-Roch theorem \cite{soule}, \cite{tamme}. 
For a motive $M=(X,p,m) \in \sM_{k,E}$, its motivic cohomology group is defined to be the $E$-module
\begin{equation*}
H^n_\sM(M,\Q(r)) = p(E \ot_\Q H_\sM^n(X,\Q(r+m))). 
\end{equation*}
For a fixed $r$, 
\begin{equation}\label{functor-motivic}
H_\sM \colon \sM_{k,E} \lra \Mod_E; \quad M \longmapsto \bigoplus_n H_\sM^n(M,\Q(r))
\end{equation}
is a well-defined covariant additive functor. 

Let $k$ be a number field. There is a functorial way \cite{scholl-2} of defining a subspace 
$$H_\sM^n(M,\Q(r))_\Z \subset H_\sM^n(M,\Q(r))$$
called the {\em integral part}. 
Conjecturally, it is a finite-dimensional $E$-vector space. 
If $M=h(X)$ and if there exists a proper flat model $\sX$ of $X$ over $\sO_k$ 
which is {\em regular}, it coincides with the original definition of Beilinson: 
$$
H_\sM^n(X,\Q(r))_\Z = \Im\bigl(K_{2r-n}(\sX)_\Q \ra K_{2r-n}(X)_\Q^{(r)}\bigr), 
$$
which is independent of the choice of $\sX$. 
The existence of such a model is known for curves.  
For a general motive $M=(X,p,m)$, we have by definition:
\begin{equation}\label{integral-part}
H^n_\sM(M,\Q(r))_\Z=\Im\bigl( E \ot_\Q H^n_\sM(X,\Q(r+m))_\Z \ra H^n_\sM(M,\Q(r))\bigr).
\end{equation}

\subsection{Deligne cohomology}

We briefly recall definitions and necessary facts on Deligne cohomology. 
For more details, see \cite{esnault-viehweg}, \cite{nekovar}, \cite{schneider}. 

For $X \in \sV_\C$, let $\Om_X^\bullet$ be the complex of the sheaves of holomorphic differential forms on $X(\C)$.  
For a subring $A \subset \R$ and an integer $r$, put
$$A(r)= (2\pi i)^rA \subset \C$$ 
and define a complex $A(r)_\sD$ of sheaves on $X(\C)$ to be 
$$A(r) \lra \sO_{X(\C)} \lra \Om_X^1  \lra \cdots \lra \Om_X^{r-1}$$
with $A(r)$ located in degree $0$. 
Then the {\em Deligne cohomology group} is defined by the hypercohomology group
\begin{equation*}
H_\sD^n(X,A(r)) = \mathbf{H}^n(X(\C),A(r)_\sD). 
\end{equation*}

By the distinguished triangle
$$\tau_{<r}\Om_X^\bullet [-1]  \lra A(r)_\sD \lra  A(r) \os{+1}{\lra}$$
we obtain a long exact sequence
\begin{multline*}
\cdots \ra H^{n-1}(X(\C),A(r)) \lra H^{n-1}_\dR(X(\C))/F^r  \\ 
\lra H_\sD^n(X,A(r)) \lra H^n(X(\C),A(r)) \lra \cdots 
\end{multline*}
where $F^\bullet$ denotes the Hodge filtration. 
If $n<2r$, then the kernel of 
$$H^n(X(\C),A(r)) \lra H^n_\dR(X(\C))/F^r$$ 
is torsion (cf. \cite{schneider}). 
Hence, for $A=\R$, we obtain 
an exact sequence
\begin{equation}\label{deligne-cohomology-exact}
0 \lra H^{n-1}(X(\C),\R(r)) \lra H^{n-1}_\dR(X(\C))/F^r   \lra H_\sD^n(X,\R(r)) \lra 0.
\end{equation}
The de Rham isomorphism together with  
the projection $\C \ra \R(r-1)$ induces an exact sequence  
\begin{equation}\label{deligne-cohomology-exact-2}
0 \lra F^rH_\dR^{n-1}(X(\C)) \lra H^{n-1}(X(\C),\R(r-1))  \lra H_\sD^n(X,\R(r)) \lra 0. 
\end{equation}

Now, let $X \in \sV_\R$. Then the {\em (real) Deligne cohomology group} of $X$ is defined by
\begin{equation*}
H^n_\sD(X, \R(r)) = H^n_\sD(X_\C,\R(r))^+, 
\end{equation*}
where ${}^+$ denotes the subspace fixed by $F_\infty \ot c_\infty$ (see \S 3.6). 
Under the GAGA isomorphism 
$H^n_\dR(X(\C)) \simeq H^n_\dR(X_\C/\C)$, 
$H^n_\dR(X(\C))^+$ corresponds to $H^n_\dR(X/\R)$, 
the algebraic de Rham cohomology of $X/\R$, 
on which the de Rham filtration is already defined.    
Therefore, if $n<2r$, then \eqref{deligne-cohomology-exact} and \eqref{deligne-cohomology-exact-2} induce 
the following exact sequences:
\begin{equation}\label{r-deligne-1}
0 \lra H^{n-1}(X(\C),\R(r))^+ \lra H^{n-1}_\dR(X/\R)/F^r  \lra H^{n}_\sD(X,\R(r)) \lra 0, 
\end{equation}
\begin{equation}\label{r-deligne-2}
0 \lra  F^rH^{n-1}_\dR(X/\R) \lra H^{n-1}(X(\C),\R(r-1))^+  \lra H^{n}_\sD(X,\R(r)) \lra 0. 
\end{equation}

Since the Deligne cohomology and homology form a twisted Poincar\'e duality theory \cite{gillet} (cf. \cite{jannsen}), 
the above definitions extend to motives: for $M=(X,p,m) \in \sM_{\R,E}$, we define 
\begin{equation*}
H_\sD^n(M,\R(r)) = p (E \ot_\Q H_\sD^n(X,\R(r+m))) 
\end{equation*}
and obtain a covariant functor
\begin{equation}\label{functor-deligne}
H_\sD \colon \sM_{\R,E} \lra \Mod_{E_\R}; \quad  M \longmapsto \bigoplus_n H^n_\sD(M,\R(r)). 
\end{equation}

\subsection{Regulator}

For  $X \in \sV_\C$, the theory of Chern characters gives 
the canonical regulator map from motivic cohomology to Deligne cohomology 
\begin{equation*}
r_\sD \colon H^n_\sM(X,\Q(r)) \lra H^n_\sD(X,\R(r))
\end{equation*} 
functorial in $X$ (see \cite{nekovar}, \cite{schneider}). 
If $X\in \sV_\R$, the image of the composite 
$$H_\sM^n(X,\Q(r)) \lra H_\sM^n(X_\C,\Q(r)) \lra H_\sD^n(X_\C,\R(r))$$ 
is contained in $H_\sD^n(X,\R(r))$. 
All these constructions extend to motives: for $M \in \sM_{\R,E}$ we have an $E$-linear map
\begin{equation*}
r_\sD \colon H^n_\sM(M,\Q(r)) \lra H^n_\sD(M,\R(r))
\end{equation*} 
functorial in $M$, i.e. compatible with \eqref{functor-motivic} and \eqref{functor-deligne}. 
 
Consider the case $n=r=2$. 
For a field $k$ and $X \in \sV_k$, let $X^{(1)}$ be the set of 
points on $X$ of codimension one.  Then we have (cf. \cite{nekovar})  
\begin{equation*}
H_\sM^2(X,\Q(2)) = \Ker\Bigl(K_2^{\mathrm{M}}(k(X))\ot \Q \os{T \ot \Q}{\lra} \bigoplus_{x \in X^{(1)}} k(x)^* \ot \Q\Bigr).  
\end{equation*}
Here, the Milnor $K$-group $K_2^\mathrm{M}(k(X))$ is the abelian group generated by symbols 
$\{f,g\} \in k(X)^* \ot_\Z k(X)^*$, divided by Steinberg relations $\{f,1-f\}=0\ (f \neq 0,1)$.  
The tame symbol $T=(T_x)$ is defined by 
\begin{equation*}
T_x(\{f,g\}) = (-1)^{\ord_x(f)\ord_x(g)} \left(\frac{f^{\ord_x(g)}}{g^{\ord_x(f)}}\right)(x). 
\end{equation*}
On the other hand, for $X \in \sV_\C$, we have by \eqref{deligne-cohomology-exact-2}
$$H^2_\sD(X,\R(2)) \os{\sim}{\lra}  H^1(X(\C),\R(1)) = \Hom(H_1(X(\C),\Z),\R(1)).$$

\begin{proposition}[\cite{beilinson-curve}, cf. \cite{ramakrishnan}]\label{formula-regulator}
Let $X$ be a smooth projective curve over $\C$ and 
 $e = \sum_i \{f_i,g_i\} \in H_\sM^2(X,\Q(2))$.  
Then, under the above identifications,  we have
\begin{equation*}
r_\sD(e)(\g)=  i \Im \sum_i \left( \int_\g \log f_i \, d\log g_i   - \log |g_i(P)| \int_\g d\log f_i \right) 
\end{equation*}
for a cycle $\g \in H_1(X(\C),\Z)$ with base point $P \in X(\C)$.  
\end{proposition}

\subsection{The Beilinson conjecture}

In the remainder of this paper, $k$ will always be a number field. 
Let $M=(X,p,m) \in \sM_{k,E}$ and assume Conjectures \ref{l-independence} and \ref{hasse-weil}. 
Recall that $L(h^i(M),s)$ is an $E_\C$-valued function. On the real axis, it takes values in 
\begin{equation}\label{E_R}
E_\R = \prod_{w |\infty}\nolimits E_w = \Bigl[\prod_{\s\colon E \hookrightarrow \C}\nolimits \C \Bigr]^+,
\end{equation}
where the script ${}^+$ denotes the fixed part by 
the complex conjugation acting both on the set $\{\s\}$ and on each $\C$.
For an integer $n$, define the {\em special value} 
$L^*(h^i(M),n) \in E_\R^*=\prod_w E_w^*$ by: 
\begin{equation*}
L^*(\s,h^i(M),n) = \lim_{s \ra n} \frac{L(\s, h^i(M),s)}{(s-n)^{\ord_{s=n} L(\s,h^i(M),s)}}.  
\end{equation*}
Note that the order of zero does not depend on $\s$. 
Moreover, Conjecture \ref{hasse-weil} and \eqref{r-deligne-1} imply that
\begin{equation*}
\rank_{E_\R} H^{i+1}_\sD(M_\R,\R(r))= \ord_{s=1-r} L(h^i(M)^\vee,s)  
\end{equation*}
if $w=i-2(m+r) \leq -3\ $  (cf. \cite{schneider}).

By composing the natural map 
$H^n_\sM(M,\Q(r))_\Z \ra H^n_\sM(M_\R,\Q(r))$
with the regulator map for $M_\R$, we obtain the regulator map for $M$
\begin{equation*}
r_\sD \colon H^n_\sM(M,\Q(r))_\Z \lra H^n_\sD(M_\R,\R(r)). 
\end{equation*} 
Let
\begin{equation*}
r_{\sD,v} \colon H^n_\sM(M,\Q(r))_\Z \lra  H^n_\sD(M_v,\R(r))
\end{equation*}
be its $v$-component. 

For an $E_\R$-module $H$, a {\em $\Q$-structure} is an $E$-submodule $H_0 \subset H$ such that 
$H_0 \ot_\Q \R =H$.  
For a ring $R$ and a free $R$-module $H$ of rank $n$, 
its {\it determinant module} is defined by 
\begin{equation*}
\det H  = \wedge^n H,
\end{equation*}
the highest exterior power. 
Let $M \in \sM_{k,E}$ and 
consider the exact sequence \eqref{r-deligne-2} for $M_\R$.  
The singular cohomology with $\R(r)$-coefficients has the natural $\Q$-structure. 
On the other hand, the de Rham cohomology has the $\Q$-structure $H^n_\dR(M/\Q)$,  
on which the Hodge filtration is already defined. 
Let 
\begin{equation*}
\sB(h^i(M)(r)) \subset \det H^{i+1}_\sD(M,\R(r)).
\end{equation*}
be the $\Q$-structure induced by \eqref{r-deligne-2}.  

\begin{conjecture}[Beilinson \cite{beilinson}]\label{conj-beilinson}
Suppose that $w =i-2(m+r)\leq -3$. 
\begin{enumerate}
\item The regulator map tensored with $\R$ 
$$r_\sD \ot_\Q \R \colon H^{i+1}_\sM(M,\Q(r))_\Z \ot_\Q \R \lra H_\sD^{i+1}(M_\R,\R(r))$$
is an isomorphism. 
\item In $\det H_\sD^{i+1}(M_\R,\R(r))$, we have 
$$
r_\sD\bigl(\det H_\sM^{i+1}(M,\Q(r))_\Z\bigr)=L^*(h^i(M)^\vee, 1-r)\sB(h^i(M)(r)).$$
\end{enumerate}
\end{conjecture}

\begin{remark}
The finite generation of the integral part of the motivic cohomology 
and the injectivity of the regulator map are in general very difficult. 
A weaker version of the conjecture is to find an 
$E$-linear subspace of $H^{i+1}_\sM(M,\Q(r))_\Z$ 
for which the same statements hold. 
The conjecture is in fact formulated for $w <0$ (see \cite{beilinson}, \cite{d-s}, \cite{nekovar}, \cite{schneider}). 
\end{remark}

In particular, if $X$ is a projective smooth curve over $k$, the conjecture (i) implies
$$\dim_\Q H_\sM^2(X,\Q(2))_\Z =[k:\Q] \cdot \mathrm{genus}(X).$$
For our Fermat motives, by the description of the Deligne cohomology which shall be given in \S 4.6,  
we should have 
$$\dim_{E_N} H_\sM^2(X_N^{a,b},\Q(2))_\Z 
=\dim_{E_{N,k}} H_\sM^2(X_N^{[a,b]_k},\Q(2))_\Z = [K_N:\Q]/2 
$$
for $(a,b) \in I_N^{\prim}$.
 
\subsection{Elements in motivic cohomology}

Starting with Ross' element, we define elements in the motivic cohomology 
of Fermat motives and study their relations. 

Let $X_N$ be the Fermat curve over a number field $k$. 
As explained in \cite{ross}, p.228, we have 
$$H_\sM^2(X_N,\Q(2))_\Z = H_\sM^2(X_N,\Q(2)),$$ 
hence by \eqref{integral-part}, we have
$$
H_\sM^2(X_N^{a,b},\Q(2))_\Z =H_\sM^2(X_N^{a,b},\Q(2)), \  
H_\sM^2(X_N^{[a,b]_k},\Q(2))_\Z = H_\sM^2(X_N^{[a,b]_k},\Q(2)). 
$$
If we put 
\begin{equation*}
e_N = \{1-x,1-y\} \in K^M_2(k(X_N)), 
\end{equation*}
then the tame symbol 
$T(e_N)$ is torsion (\cite{ross}, Theorem 1),  
so $e_N$ defines an element of $H^2_\sM(X_N,\Q(2))_\Z$. 
\begin{remark}
The divisors of $1-x$, $1-y$ and their $G_N$-translations are supported on torsion points of $X_N$, embedded in its Jacobian variety by choosing as a base point any point with $x_0y_0z_0=0$. 
\end{remark}

\begin{definition}
Define{\rm :} 
\begin{alignat*}{2}
&e_N^{a,b} = p_N^{a,b} \pi_{K_N/k}^* e_N & & \in H_\sM^2(X_N^{a,b},\Q(2))_\Z, \\
&e_N^{[a,b]_k} =p_N^{[a,b]_k} e_N &&\in H_\sM^2(X_N^{[a,b]_k},\Q(2))_\Z. 
\end{alignat*}
\end{definition}

\begin{proposition}\label{e_N} 
If $N=N'd$, $(a,b) \in G_N$ and $(a',b')\in G_{N'}$, then we have{\rm :}
\begin{enumerate}
	\item $\pi_{N/N',k *}e_N = e_{N'}$. 
	\item $\pi_{N/N',K_N/K_{N'}}^*e_{N'}^{a',b'} = d^2 e_N^{a'd,b'd}$. 
	\item $\pi_{N/N',K_N *} e_N^{a,b} = \begin{cases}
			\pi_{K_N/K_{N'}}^* e_{N'}^{a',b'} & \text{if $(a,b)=(a'd,b'd), \exists(a',b')\in G_{N'}$}, \\
						0 & \text{otherwise}. 								
						\end{cases}$ 
	\item  $\pi_{N/N',k}^* e_{N'}^{[a',b']_k} = d^2 e_N^{[a'd,b'd]_k}$. 
	 \item $\pi_{N/N',k *} e_N^{[a,b]_k} 
	= \begin{cases} e_{N'}^{[a',b']_k} & \text{if $(a,b)=(a'd,b'd), \exists(a',b') \in G_{N'}$}, \\
	0 & \text{otherwise}. \end{cases}$
	\item $\pi_{N,K_N/k}^* e_N^{[a,b]_k} = \sum_{(c,d)\in[a,b]_k} e_N^{c,d}$. 
	\item $\pi_{N,K_N/k *} e_N^{a,b} = \frac{[K_N:k]}{\sharp[a,b]_k}e_N^{[a,b]_k} \ (=e_N^{[a,b]_k} \ 
	\text{if} \ (a,b)\in I_N^\prim)$. 
\end{enumerate}
\end{proposition}

\begin{proof} 
(i) \ Let $(x,y)$ (resp. $(x',y')$) be the affine coordinates of $X_N$ (resp. $X_{N'}$), so that $x'=x^d$, $y'=y^d$. 
Consider the intermediate curve 
$$X_{N,N'} \colon x^N+y'^{N'}=1,$$
with natural morphisms $X_N \os{\pi_1}{\lra} X_{N,N'} \os{\pi_2}{\lra} X_{N'}$. 
By the projection formula for the cup product in $K$-theory,  we have: 
\begin{align*}
&\pi_{N/N' *}\{1-x,1-y\} = \pi_{2 *} \pi_{1 *}\{\pi_1^*(1-x),1-y\}\\ 
& = \pi_{2 *} \{1-x,\pi_{1 *}(1-y)\} = \pi_{2 *}\{1-x,1-y'\} \\
& = \pi_{2 *} \{1-x, \pi_2^*(1-y')\} = \{\pi_{2 *}(1-x), 1-y'\} \\& = \{1-x',1-y'\}. 
\end{align*}

(ii) \ Since 
$$\pi_{N/N',K_N/k}^*e_{N'} = \left(\sum_{g \in G_{N/N'}}\nolimits g\right) \pi_{N,K_N/k}^*e_N$$ 
and $p_N^{a'd,b'd}g=p_N^{a'd,b'd}$ for $g \in G_{N/N'}$, this follows from the commutativity of \eqref{four-1}. 

(iii) \  The first case follows from (i) and the commutativity of \eqref{four-1}. 
For the second case, $\pi_{N/N',K_N}^*$ is injective and we have 
\begin{align*}
\pi_{N/N',K_N}^*  \pi_{N/N',K_N *} p_N^{a,b}
= \sum_{g \in G_{N/N'}} g p_N^{a,b}
= \sum_{g \in G_{N/N'}} \theta_N^{a,b}(g) p_N^{a,b}=0. 
\end{align*}
(iv) and (v) follow similarly as (ii) and (iii) from the commutativity of \eqref{four-2}. 
(vi) is clear by definition. 

(vii) \ Using (ii), we are reduced to the primitive case, which follows from \eqref{pi-p-pi}. 
\end{proof}

\subsection{Deligne cohomology of Fermat motives}

We calculate the Deligne cohomology of $X_N^{a,b} \in \sM_{K_N,E_N}$. 
Note that both $K_N$ and $E_N$ are totally imaginary. 

Let $M \in \sM_{k,E}$ be a motive, and for a complex place $v$ of $k$, 
let $\t, \ol\t \colon k \hookrightarrow \C$ be the conjugate embeddings inducing $v$, 
and put $M_\t = \vphi_{\C/k,\t}^*M$. 
Since $F_\infty$ exchanges the components of 
$$H_\sD^n(M_{v,\C},\R(r)) = H_\sD^n(M_\t,\R(r)) \times H_\sD^n(M_{\ol\t},\R(r)), $$
we have canonically 
\begin{equation*}
H_\sD^n(M_v,\R(r)) = H_\sD^n(M_\t,\R(r)).
\end{equation*}
In particular, for each infinite place $v$ of $K_N$ and a choice of $\tau$, 
we have an identification of $E_{N,\R}$-modules  
\begin{equation}\label{identification}
H^2_\sD(X_{N,v}^{a,b},\R(2))
=  H^2_\sD(X_{N,\t}^{a,b},\R(2)) = H^1(X_{N,\t}^{a,b},\R(1)).   
\end{equation}
The $\Q$-structure $\sB$ splits as
$$\sB(h^1(X_N^{a,b})) = \bigoplus_{v | \infty} \sB(h^1(X_{N,v}^{a,b}))$$
where $\sB(h^1(X_{N,v}^{a,b}))$ 
 corresponds via \eqref{identification} to $H^1(X_{N,\t}^{a,b},\Q(1))$. 
 
Similarly to the $\ell$-adic case, for an $E_\R$-module $V$, let 
$$V = \bigoplus_{w | \infty} V_w, \quad V_w=E_w \ot_{E_\R} V$$ 
be the decomposition corresponding to \eqref{E_R}.   
If $w$ is a complex place and $\s,\ol\s \colon E \hookrightarrow \C$ 
are the embeddings inducing $w$, then we have 
$$V_w = \bigl[V_\s \oplus V_{\ol\s}\bigr]^+,$$
where we put $V_\s= \C \ot_{E_\R, \s} V$, and $+$ denotes the part fixed by the complex conjugation acting 
both on the set $\{\s,\ol\s\}$ and on $\C$. 
Therefore we have a canonical isomorphism $V_w = V_\s$. 
For $v \in V$, let $v_\s \in V_\s$ denote its $\s$-component. 

Applying these to our situation, for each infinite place $v$ of $K_N$ 
and an embedding $\s\colon E_N \hookrightarrow \C$, 
we obtain an identification
\begin{equation}\label{identification-sigma}
H_\sD^1(X_{N,v}^{a,b},\R(2))_\s = H^1(X_{N,\t}^{a,b}(\C),\R(1))_\s = \s(p_N^{a,b})H^1(X_{N,\t}(\C),\C),
\end{equation}
the subspace on which $G_N$ acts by the $\C^*$-valued character $\s\theta_N^{a,b}$.

Now, for the moment, let $X_N$ be the Fermat curve over $\C$. 
By choosing a primitive root of unity $\z_N \in \C$, $G_N$ acts on $X_N$.  
Let us recall the structure of the homology and cohomology groups of $X_N(\C)$. 
See \cite{otsubo}, \cite{rohrlich} for the details.  
\begin{definition}
Define a path by
$$\d_N\colon [0,1] \lra X_N(\C); \quad t \longmapsto (t^{\frac 1 N},(1-t)^{\frac 1 N})$$
where the branches are taken in $\R$. 
Then, 
$(1-g_N^{r,0})(1-g_N^{0,s})\d_N$ becomes a cycle and defines an element of $H_1(X_N(\C),\Q)$. 
Put
$$\g_N = \frac{1}{N^2} \sum_{(r,s) \in G_N} (1-g_N^{r,0})(1-g_N^{0,s})\d_N.$$
It does not depend on the choice of $\z_N$. 
\end{definition}

\begin{definition}
For integers $a$, $b$, define a differential form on $X_N(\C)$ by
\begin{equation*}
\om_N^{a,b} = x^ay^{b-N}\frac{dx}{x} = -x^{a-N}y^b \frac{dy}{y}.
\end{equation*}
For $(a,b) \in G_N$, put $\om_N^{a,b} = \om_N^{\langle a \rangle, \langle b \rangle}$, 
where $\langle a \rangle \in \{1,2,\dots, N\}$ denotes the integer representing $a$. 
If $(a,b) \in I_N$, then $\om_N^{a,b}$ is of the second kind (i.e. has no residues), 
so defines an element of $H^1(X_N(\C),\C)$, which we denote by the same letter. 
Moreover, $\om_N^{a,b}$ is of the first kind (i.e. holomorphic) if and only if $\angle{a}+\angle{b}<N$. 
\end{definition}

\begin{proposition}\label{homology}  \
\begin{enumerate}
\item $H_1(X_N(\C),\Q)$ is a cyclic $\Q[G_N]$-module generated by $\g_N$. 
\item The set $\bigl\{\om_N^{a,b} \bigm| (a,b) \in I_N\bigr\}$ is a basis of $H^1(X_N(\C),\C)$. 
\item For $(a,b) \in I_N$, we have 
$$\int_{\g_N} \om_N^{a,b} = \frac{1}{N} B\bigl(\tfrac{\angle a}{N},\tfrac{\angle b}{N}\bigr),$$
where $B(\a,\b)$ is the Beta function.  
\end{enumerate}
\end{proposition}
Note that $\om_N^{a,b}$ is an eigenform for the $G_N$-action: 
$$g_N^{r,s}\om_N^{a,b} = \z_N^{ar+bs} \om_N^{a,b}.$$
We normalize $\om_N^{a,b}$ as 
\begin{equation}\label{omega-normalized}
\wt\om_N^{a,b} := \Bigl(\frac 1 N B\bigl(\tfrac{\angle a}{N},\tfrac{\angle b}{N}\bigr)\Bigr)^{-1}  \om_N^{a,b}.
\end{equation}
Then we have for any $g_N^{r,s} \in G_N$
\begin{equation}\label{period}
\int_{g_N^{r,s}\g_N}\wt\om_N^{a,b}=\int_{\g_N}g_N^{r,s}\wt\om_N^{a,b}=\z_N^{ar+bs}.
\end{equation}
Hence we have 
\begin{equation*}
c_\infty \wt\om_N^{a,b} =\wt\om_N^{-a,-b}.
\end{equation*}

Let us return to the original situation over $K_N$, and for each embedding $\tau \colon K_N \hookrightarrow \C$, 
let 
\begin{equation*}
\g_{N,\t} \in H_1(X_{N,\t}(\C),\Q), \quad \om_{N,\t}^{a,b},  \wt\om_{N,\t}^{a,b} \in H^1(X_{N,\t}(\C),\C)
\end{equation*}
be the corresponding classes for $X_{N,\t}(\C) = \Mor_{K_N,\t}(\C,X_{N,K_N})$. 
By Proposition \ref{homology} (i), it follows that 
$$H_1(X_{N,\t}^{a,b}(\C),\Q) := p_N^{a,b} (E_N \ot_\Q H_1(X_{N,\t}(\C),\Q))$$
is a one-dimensional $E_N$-vector space generated by $p_N^{a,b}\g_{N,\t}$. 

\begin{definition}
For each infinite place $v$ of $K_N$, choose $\t$ inducing $v$. For $(a,b) \in I_N$, 
define 
$$\l_{N,v}^{a,b} \in H_\sD^2(X_{N,v}^{a,b},\R(2))$$ 
to be the element corresponding to 
$2\pi i \cdot(p_N^{a,b}\g_{N,\t})^\vee$ under the identification \eqref{identification}. 
Only the sign depends on the choice of $\t$. 
\end{definition}

\begin{proposition}\label{basis-deligne}
Let $(a,b) \in I_N$ and the notations be as above. Then, 
\begin{enumerate}
\item 
$H_\sD^2(X_{N,v}^{a,b},\R(2))=E_{N,\R} \l_{N,v}^{a,b}$ and 
$\sB(h^1(X_{N,v}^{a,b}))=E_N \l_{N,v}^{a,b}$. 
\item
Under the identification \eqref{identification-sigma},  we have
$$(\l_{N,v}^{a,b})_\s = 2 \pi i \cdot \wt\om_{N,\t}^{ha,hb}$$ 
where $h \in H_N$ is the element such that 
$\t(\z_N)^h=\s(\xi_N)$. 
\end{enumerate}
\end{proposition}

\begin{remark}\label{remark-h-number} 
The equality \eqref{h-number} follows since
$$H^1(X_{N,v}^{a,b}(\C),\C) = H^1(X_{N,\t}^{a,b}(\C),\C) \oplus H^1(X_{N,\ol\t}^{a,b}(\C),\C)$$
and exactly one of $\om_{N,\t}^{ha,hb}$, $\om_{N,\ol\t}^{-ha,-hb}$ is holomorphic. 
\end{remark}

\subsection{Main results}
We state the main results of this paper. 

For $\a \in \C - \{0,-1,-2, \dots\}$ and a non-negative integer $n$, let 
\begin{equation*}
(\a,n)=\a(\a+1)(\a+2)\cdots (\a+n-1) = \frac{\vG(\a+n)}{\vG(\a)}
\end{equation*}
be the Pochhammer symbol, where $\vG(\a)$ is the Gamma function. 

\begin{definition}\label{def-F}
Define a function of positive real numbers $\a$, $\b$, by
$$\wt F(\a,\b)= \frac{\vG(\a)\vG(\b)}{\vG(\a+\b+1)} \sum_{m, n \geq 0} \frac{(\a,m)(\b,n)}{(\a+\b+1,m+n)},$$
which takes values in positive real numbers. Its convergence will be explained later. 
\end{definition}

The main result of this paper is the following.
 
\begin{theorem}\label{main-theorem} 
Let $(a,b) \in I_N$, and $v$ be an infinite place of $K_N$ induced by $\t$. 
Consider the regulator map
$$r_{\sD,v} \colon H_\sM^2(X_N^{a,b},\Q(2))_\Z  \lra H_\sD^2(X_{N,v}^{a,b},\R(2)).$$
Then we have
$$r_{\sD,v}(e_N^{a,b}) = \mathbf{c}_{N,v}^{a,b} \l^{a,b}_{N,v}$$
with $\mathbf{c}_{N,v}^{a,b} \in E_{N,\R}^*$. 
For any embedding 
$\s \colon E_N \hookrightarrow \C$, we have
$$\s(\mathbf{c}_{N,v}^{a,b}) = -\frac{1}{4N^2\pi i}
\Bigl(\wt F\bigl(\tfrac{\angle{ha}}{N},\tfrac{\angle{hb}}{N}\bigr)-\wt{F}\bigl(\tfrac{\angle{-ha}}{N},\tfrac{\angle{-hb}}{N}\bigr) \Bigr) \ \in \C^* $$
where $h \in H_N$ is the unique element satisfying $\t(\z_N)^h=\s(\xi_N)$. 
In particular, $r_{\sD,v} \ot_\Q \R$ is surjective. 
\end{theorem}

The proof will be given in the subsequent subsections. First, we give several corollaries. 

\begin{corollary}\label{corollary-1}
If $k$ contains all the $N$-th roots of unity, then the regulator map 
$$r_{\sD,v} \ot_\Q \R  \colon H_\sM^2(X_N,\Q(2))_\Z \ot_\Q \R \lra H_\sD^2(X_{N,v},\R(2))$$
is surjective for any infinite place $v$ of $k$. 
\end{corollary}
\begin{proof}
After tensoring with $E_N$, the both sides decomposes into the cohomology groups of $X_N^{a,b}$, 
on which the regulator is surjective by Theorem \ref{main-theorem}. 
Hence the original map is surjective. 
\end{proof}

\begin{corollary}\label{corollary-2}
Let $(a,b) \in I_N$ and $v$ be an infinite place of $k$. Then the image of 
$e_N^{[a,b]_k}$ under the regulator map 
\begin{align*}
 r_{\sD,v} \colon H_\sM^2(X_N^{[a,b]_k},\Q(2))_\Z & \lra H_\sD^2(X_N^{[a,b]_k},\R(2))
\end{align*}
is non-trivial. 
\end{corollary}
\begin{proof}
By Proposition \ref{prop-level} (ii) and Proposition \ref{e_N} (iv), 
we can assume that $(a,b)$ is primitive. 
By Proposition \ref{prop-field} (ii), 
after taking $E_N \ot_{E_{N,k}} -$, the regulator in question is identified 
with the product of the regulators of Theorem \ref{main-theorem} for the places of $K_N$ over $v$, 
under which $e_N^{[a,b]_k}$ corresponds to  
$e_N^{a,b}$ by Proposition \ref{e_N} (vii).  
\end{proof}

\begin{corollary}\label{corollary-3} 
Suppose that $N=3$, $4$ or $6$, and $k\subset\Q(\z_N)$. Then the regulator map  
\begin{align*}
r_\sD \ot_\Q \R \colon H^2_\sM(X_N,\Q(2))_\Z \ot_\Q \R \lra H^2_\sD(X_{N,\R},\R(2))
\end{align*}
is surjective.  
\end{corollary}

\begin{proof}
Since $\Q(\mu_N)$ is imaginary quadratic, 
the surjectivity for $X_{N,\Q(\mu_N)}$ (resp. $X_{N,\Q}$) follows from Corollary \ref{corollary-1} 
(resp. Corollary \ref{corollary-2}). 
\end{proof}

\begin{remark} 
If $N=3$, $4$ or $6$, then
the motive $X_N^{[a,b]_\Q}$ is isomorphic to $h^1(E)$, 
where $E$ is an elliptic curve over $\Q$ with complex multiplication by the integer ring of $\Q(\mu_N)$. 
Therefore, the surjectivity was already known \cite{beilinson}, \cite{bloch}, \cite{deninger} (see also \cite{d-w}). 
\end{remark}

\subsection{Calculation of the regulators}

We calculate the regulator of $e_N^{a,b}$ and prove the formula of Theorem \ref{main-theorem}. 
First, since
$$d\log (1-x) = - \sum_{m \geq 1} x^m \frac{dx}{x}, \quad \log (1-y) = - \sum_{n \geq 1} \frac{y^n}{n},$$
we have: 
$$d\log (1-x)\log (1-y) = \sum_{m,n \geq 1} \frac 1 n x^my^n \frac{dx}{x} 
= \sum_{m,n \geq 1} \frac 1 n \om_N^{m,n+N}. $$

\begin{lemma} For $a$, $b \in \Z$, we have modulo exact forms
$$(\tfrac a N + \tfrac b N,i+j)\om_N^{a+Ni,b+Nj} \equiv (\tfrac a N,i)(\tfrac b N,j)\om_N^{a,b}.$$
\end{lemma}

\begin{proof}
First, since 
$$
d(x^ay^b) = ax^ay^b\frac{dx}{x} + bx^ay^b\frac{dy}{y} = ax^ay^b\frac{dx}{x} - bx^{a+N}y^{b-N}\frac{dx}{x},   
$$
we have $a \om_N^{a,b+N} \equiv b \om_N^{a+N,b}$. On the other hand, 
\begin{align*}
\om_N^{a+N,b} = x^a(1-y^N)y^{b-N}\frac{dx}{x} = \om_N^{a,b}-\om_N^{a,b+N}.   
\end{align*}
From these we obtain 
$$(a+b)\om_N^{a+N,b} \equiv a \om_N^{a,b}, \quad (a+b)\om_N^{a,b+N} \equiv b \om_N^{a,b}.$$ 
By using these formulae repeatedly, we obtain the result. 
\end{proof}

\begin{remark}
This relation reflects, and is in fact equivalent to, the relation of Beta values: 
$$(\tfrac a N+ \tfrac b N,i+j)B(\tfrac a N+i,\tfrac b N+j) = (\tfrac a N,i)(\tfrac b N,j)B(\tfrac a N, \tfrac b N),$$
which follows from the well-known relations
\begin{equation}\label{beta-gamma}
B(\a,\b)= \frac{\vG(\a)\vG(\b)}{\vG(\a+\b)}, \quad \vG(\a+1)=\a\vG(\a).
\end{equation}
\end{remark}

Using the above lemma, we obtain:  
\begin{equation*}\begin{split} 
d\log&(1-x)\log(1-y) \\
&\equiv \sum_{m,n \geq 1} \frac{1}{m+n} \, \om_N^{m,n}\\
& = \sum_{1 \leq a,b \leq N} \sum_{i,j \geq 0} \frac{1}{a+Ni+b+Nj} \, \om_N^{a+Ni,b+Nj} \\
& \equiv  \frac{1}{N} \sum_{1 \leq a,b \leq N} \sum_{i,j \geq 0} \frac{(\frac{a}{N},i)(\frac{b}{N},j)}{(\frac{a}{N}+\frac{b}{N}, i+j+1)} \, \om_N^{a,b}\\
&= \frac{1}{N^2}  \sum_{1 \leq a,b \leq N} \frac{\vG(\frac a N)\vG(\frac b N)}{\vG(\frac{a}{N}+\frac{b}{N}+1)}
 \sum_{i,j \geq 0} \frac{(\frac{a}{N},i)(\frac{b}{N},j)}{(\frac{a}{N}+\frac{b}{N}+1, i+j)} \, 
 \wt\om_N^{a,b}\\
& = \frac{1}{N^2} \sum_{1 \leq a,b \leq N} \wt F(\tfrac a N, \tfrac b N)\, \wt\om_N^{a,b} . 
\end{split}\end{equation*}

We apply Proposition \ref{formula-regulator} for $f=1-y$, $g=1-x$; note that $e_N= -\{f,g\}$. 
We can start our cycles $(1-g_N^{r,0})(1-g_N^{0,s})\d_N$ from 
$P=(0,1)$, so that the second term of the formula vanishes. 
(More precisely, we modify slightly the cycle so that it is contained in the region $|y|<1$, 
and does not pass through the singularities of $f$ and $g$.)   

Now we calculate the first term of the formula.  
First, since $\om_N^{a,N}$, $\om_N^{N,b}$ are exact forms, they have trivial periods. 
Secondly, let $a+b=N$. Then $\om_N^{a,b}$, only having logarithmic singularities along $Z_N(\C)$, 
is a well-defined element of $H^1(U_N(\C),\C)$. 
Our cycles $g_N^{r,s}\g_N$ are already defined on $U_N(\C)$, and the formula \eqref{period} holds also in this case. 
Since $\wt F(\tfrac a N, \tfrac b N) = \wt F(\tfrac b N, \tfrac a N)$, and  
$$\int_{g_N^{r,s}\g_N} \bigl(\wt\om_N^{a,b}+\wt\om_N^{b,a}\bigr) = \z_N^{a(r-s)}+\z_N^{a(s-r)}$$
is a real number, these terms do not contribute to the regulator.  

Therefore, for a cycle $\g' \in H_1(X_{N,\t}(\C),\Q)$, we obtained: 
\begin{equation}\label{regulator-step2}
\begin{split}
r_{\sD,v}(e_N)(\g') 
&= -\frac{i}{N^2} \Im \left(\sum_{(a,b)\in I_N} \wt F\bigl(\tfrac{\angle{a}}{N}, \tfrac{\angle{b}}{N}\bigr) 
\int_{\g'} \wt\om_{N,\t}^{a,b} \right)\\
&= -\frac{1}{2N^2}  \sum_{(a,b)\in I_N} \wt F\bigl(\tfrac{\angle{a}}{N}, \tfrac{\angle{b}}{N}\bigr)
\int_{\g'} \bigl(\wt\om_{N,\t}^{a,b}-\wt\om_{N,\t}^{-a,-b}\bigr)\\
&= -\frac{1}{2N^2}  \sum_{(a,b)\in I_N} \Bigl(\wt F\bigl(\tfrac{\angle{a}}{N}, \tfrac{\angle{b}}{N}\bigr)
-\wt F\bigl(\tfrac{\angle{-a}}{N}, \tfrac{\angle{-b}}{N}\bigr)\Bigr) \int_{\g'} \wt\om_{N,\t}^{a,b}. 
\end{split}\end{equation}
Apply this to $\g' = p_N^{a,b}\g_{N,\t}$.  
By the adjointness 
$$\int_{p_N^{a,b}\g_{N,\t}}\wt\om_{N,\t}^{c,d}=\int_{\g_{N,\t}} p_N^{a,b}\wt\om_{N,\t}^{c,d},$$
and Proposition \ref{basis-deligne}, we obtain the formula of Theorem \ref{main-theorem}. 
We are left to show that $\mathbf{c}_{N,v}^{a,b}$ is invertible, which will be done in the next subsection.

\begin{corollary}\label{346}
Let $N=3, 4$, or $6$, $(a,b) \in I_N$, and assume the Beilinson conjecture (Conjecture \ref{conj-beilinson}) 
for $X_N^{[a,b]_\Q} \in \sM_\Q$ . 
Then it follows
\begin{equation*}
 L(j_N^{a,b},2)\equiv \pi L^*(j_N^{a,b},0) \equiv
 \sin \tfrac{2\pi}{N} 
\Bigl(\wt F\bigl(\tfrac{\angle{a}}{N},\tfrac{\angle{b}}{N}\bigr)-\wt{F}\bigl(\tfrac{\angle{-a}}{N},\tfrac{\angle{-b}}{N}\bigr) \Bigr)  
\end{equation*}
modulo $\Q^*$.
\end{corollary}

\begin{proof}
The first equivalence follows from Remark \ref{sigma-independence}, Corollary \ref{functional-equation} and Remark \ref{epsilon}.

We calculate the regulator of $e_N^{[a,b]_\Q}$. 
The target of the regulator is 
$$H^1(X_N(\C),\Q(1))^+ = \Hom\bigl(H_1(X_N(\C),\Q)^-,\Q(1)\bigr).$$
Since $F_\infty \d_N=\d_N$, we have $F_\infty g_N^{r,s}\d_N = g_N^{-r,-s}\d_N$, 
$F_\infty\g_N=\g_N$, and hence $F_\infty g_N^{r,s}\g_N=g_N^{-r,-s}\g_N$. 
By Proposition \ref{homology} (i), 
$H_1(X_N(\C),\Q)^-$ is generated by 
$$\bigl\{(g_N^{r,s}-g_N^{-r,-s})\g_N \bigm| (r,s) \in G_N\bigr\}.$$ 

Since the only non-primitive case is $N=6$, $[a,b]_\Q=[2,2]_\Q$, which reduces to $N=3$, $[a,b]_\Q=[1,1]_\Q$, we can assume that $(a,b)$ is primitive. 
Choose $(r,s)$ such that $ar+bs=1$. Then it follows that
$H_1(X_{N}^{[a,b]}(\C),\Q)^-$ 
is a one-dimensional $\Q$-module generated by
$$\g_N^{[a,b]_\Q}:= 
p_N^{[a,b]_\Q}(g_N^{r,s}-g_N^{-r,-s})\g_N = (\xi_N-\xi_N^{-1})(p_N^{a,b}-p_N^{-a,-b})\g_N.$$
Therefore, by \eqref{regulator-step2} we have  
$$r_{\sD}(e_N^{[a,b]_\Q})(\g_N^{[a,b]_\Q})=-\frac{1}{N^2}(\xi_N-\xi_N^{-1})
\Bigl(\wt F\bigl(\tfrac{\angle{a}}{N},\tfrac{\angle{b}}{N}\bigr)-\wt{F}\bigl(\tfrac{\angle{-a}}{N},\tfrac{\angle{-b}}{N}\bigr)\Bigr),$$
hence the second equivalence follows. 
\end{proof}

\begin{remark}
Some cases of the corollary are proved unconditionally (with the rational factor determined) in \cite{otsubo-comparison} 
by comparing our element $e_N^{[a,b]_\Q}$ with Bloch's element \cite{bloch} for an elliptic curve with complex multiplication.    
\end{remark}

\subsection{Hypergeometric functions and the end of the proof}

We introduce Appell's hypergeometric function $F_3$, and finish the proof of 
Theorem \ref{main-theorem}. 

First, let us recall some properties of the classical hypergeometric series of Gauss 
\begin{equation*}
F(\a,\b,\g; x) = \sum_{n\geq 0}\frac{(\a,n)(\b,n)}{(\g,n)(1,n)} x^n,  
\end{equation*}
where $\g \not\in\{0, -1, -2, \dots\}$. 

\begin{proposition}[cf. \cite{t-kimura}]\label{gauss} \ 
\begin{enumerate}
\item $F(\a,\b,\g;x)$ converges absolutely for $|x|<1$. 
\item If $|x|<1$ and $\Re(\g)>\Re(\a)>0$, then we have:  
\begin{equation*}
F(\a,\b,\g;x) = \frac{\vG(\g)}{\vG(\a)\vG(\g-\a)} \int_0^1 u^{\a-1}(1-u)^{\g-\a-1}(1-xu)^{-\b} \, du,
\end{equation*}
where the integral is taken along the segment $0 \leq u \leq 1$, and the branches are determined by 
$\arg(u)=0$, $\arg(1-u)=0$ and $|\arg(1-xu)| \leq \pi/2$. 
\item If\, $\Re(\g-\a-\b)>0$, then $F(\a,\b,\g;x)$ converges ansolutely for $|x|=1$,  and we have
\begin{equation*}
F(\a,\b,\g;1)= \frac{\vG(\g)\vG(\g-\a-\b)}{\vG(\g-\a)\vG(\g-\b)}.
\end{equation*}
As a function of $\a$, $\b$ and $\g$, $F(\a,\b,\g;1)$ is holomorphic in the domain $\Re(\g-\a-\b)>0$, 
$\g \not\in\{0, -1, -2, \dots\}$. 
\end{enumerate}
\end{proposition}

Appell's hypergeometric series $F_3(\a,\a'\b,\b',\g;x,y)$ of two variables is defined for $\g \not\in\{0, -1, -2, \dots\}$ 
by 
\begin{equation*}
F_3(\a,\a',\b,\b',\g;x,y) = \sum_{m,n \geq 0} \frac{(\a,m)(\a',n)(\b,m)(\b',n)}{(\g,m+n)(1,m)(1,n)} x^my^n. 
\end{equation*}
This satisfies the following properties: 
\begin{proposition}\label{F_3} \ 
\begin{enumerate}
	\item $F_3(\a,\a',\b,\b',\g;x,y)$ converges absolutely for $|x|<1$, $|y|<1$. 
	\item If $\Re(\a)>0$, $\Re(\a')>0$, and  $\Re(\g-\a-\a')>0$, then we have: 
\begin{multline*}
F_{3}(\a,\a',\b,\b',\g; x,y) = \frac{\vG(\g)}{\vG(\a)\vG(\a')\vG(\g-\a-\a')} \\
 \times \iint_{\Delta}u^{\a-1}(1-xu)^{-\b}v^{\a'-1}(1-yv)^{-\b'}(1-u-v)^{\g-\a-\a'-1} du \, dv,  
\end{multline*}
where $\Delta = \{(u,v) \mid u, v, 1-u-v \geq 0\}$, and the branches of the integrands are chosen similarly as above. 	
	\item Suppose that $\Re(\g-\a-\b)>0$ and $\Re(\g-\a'-\b')>0$. Then, $F_3(\a,\a',\b,\b',\g;x,y)$ converges absolutely for $|x|=|y|=1$. As a function of $\a$, $\a'$, $\b$, $\b'$ and $\g$, $F_3(\a,\a',\b,\b',\g;1,1)$ is holomorphic in the domain: $\Re(\g-\a-\b)>0$, $\Re(\g-\a'-\b')>0$, $\g \neq 0, -1, -2, \dots$. 
\end{enumerate}
\end{proposition}

\begin{proof}
We only prove (iii). See \cite{appell}, \cite{t-kimura} for the other statements.  First, we have 
$$F_3(\a,\a',\b,\b',\g;x,y) = \sum_{n \geq 0} \frac{(\a',n)(\b',n)}{(\g,n)(1,n)} F(\a,\b,\g+n;x) y^n. $$
Since $\Re(\g+n-\a-\b)>0$, $F(\a,\b,\g+n;x)$ converges absolutely for $|x|=1$ by Proposition \ref{gauss}. 
Since $|\g+n|>|\g|$ for sufficiently large $n$, and then $|\g+n+i|>|\g+i|$ for any $i$, the sum
$$\sum_{m \geq 0} \left|\frac{(\a,m)(\b,m)}{(\g+n,m)(1,m)}\right|$$
is bounded independently of $n$, and the absolute convergence of $F_3$ follows from that of 
$F(\a',\b',\g;y)$ for $|y|=1$, which follows from the assumption and 
Proposition \ref{gauss}. 

The holomorphicity follows by a similar argument as in the case of one variable using the fact that 
$F_3$ is a Newton series with respect to $\a$, $a'$, $\b$ and $\b'$, and is a factorial series with respect to $\g$. 
\end{proof}

Now, consider the special case  
\begin{equation*}
F(\a,\b;x,y) := F_3(\a,\b,1,1,\a+\b+1;x,y)
= \sum_{m,n \geq 0} \frac{(\a,m)(\b,n)}{(\a+\b+1,m+n)} x^my^n. 
\end{equation*}
If $\Re(\a), \Re(\b) >0$, then by the above proposition, it converges absolutely for $|x|, |y| \leq 1$, 
and the integral representation takes the form 
$$F(\a,\b; x,y) = \frac{\vG(\a+\b+1)}{\vG(\a)\vG(\b)} 
\iint_\Delta u^{\a-1} (1-xu)^{-1} v^{\b-1} (1-yv)^{-1} du \, dv.$$
In particular, if we define (see Definition \ref{def-F})
\begin{equation*}
\wt{F}(\a,\b) = \frac{\vG(\a)\vG(\b)}{\vG(\a+\b+1)} F(\a,\b;1,1), 
\end{equation*}
then we have
\begin{equation}\label{equation-integral}
\wt{F}(\a,\b)=\iint_{\Delta} u^{\a-1}(1-u)^{-1}v^{\b-1}(1-v)^{-1} du \, dv.
\end{equation}

\begin{proposition}\label{decreasing}
Consider $\wt{F}(\a,\b)$ as a function of positive real numbers $\a$, $\b$. Then{\rm :} 
\begin{enumerate}
	\item $\wt{F}(\a,\b)$ is monotonously decreasing with respect to each parameter. 
	\item Suppose that  $0 < \a, \b <1$. Then, $\wt{F}(\a,\b) \neq \wt{F}(1-\a,1-\b)$ 
	if and only if $\a+\b\neq 1$.
\end{enumerate}	
\end{proposition}

\begin{proof}
(i) is immediate from \ref{equation-integral}. To prove (ii), first assume that $\a+\b<1$. Then we have
$$\wt{F}(\a,\b) = \wt{F}(\b,\a) > \wt{F}(1-\a,\a) >\wt{F}(1-\a,1-\b).$$ Similarly, if $\a+\b>1$, 
then $\wt{F}(\a,\b)<\wt{F}(1-\a,1-\b)$. Finally, if $\a+\b=1$, then we have 
$\wt{F}(\a,\b)=\wt{F}(\b,\a)=\wt{F}(1-\a,1-\b)$. 
\end{proof}

Applying this proposition to $\a = \frac{\angle{ha}}{N}$, $\b = \frac{\angle{hb}}{N}$, the proof of Theorem \ref{main-theorem} is completed. 

\begin{remark}
We could also use Deligne's description of the regulator (cf. \cite{hain} \cite{ramakrishnan}), 
which uses  Chen's iterated integral.  
For a path $\g \colon [0,1] \ra X$ and differential forms $\omega$, $\eta$, we define 
$$\int_\g \om\eta := \int_0^1\left(\int_0^t \g^*\om(s)\right) \g^*\eta(t).$$
Roughly speaking, the regulator map sends $\{f,g\}$ to 
$$\g \longmapsto  i \Im\left( \int_\g d\log f \, d\log g\right). $$
In this way, one finds more directly the integral \eqref{equation-integral}. 
Note that the iterated integral is a double integral over the region $0 \leq s, t, t-s \leq 1$, which is transformed to 
$\Delta$ by $u=s$, $v=1-t$.  
\end{remark}

\subsection{Variants}

We discuss some variants which involve special values of hypergeometric functions of one variable (cf. \cite{slater})
$${}_{p}F_{q}\left({{\a_1,\cdots, \a_{p}} \atop {\b_1,\cdots, \b_{q}}};x\right)
= \sum_{n \geq 0} \frac{(\a_1,n)\cdots (\a_p,n)}{(\b_1,n)\cdots (\b_q,n)(1,n)}x^n. $$ 
It converges absolutely for $|x|<1$, and converges at $x=1$ if 
$$\sum_{j=1}^q \b_j - \sum_{i=1}^p \a_i >0.$$  
The integral representation of ${}_3F_2$ (cf. \cite{slater}, 4.1) is written as follows: 
\begin{multline*}
{}_3F_2\left({{a,b,c} \atop {d,e}};x\right) 
= \frac{\vG(d)\vG(e)}{\vG(a)\vG(d-a)\vG(c)\vG(e-c)} \times \\ 
\iint_{\Delta} u^{a-1}(1-xu)^{-b}(1-u(1-v)^{-1})^{d-a-1}v^{e-c-1}(1-v)^{c-a-1} du \, dv. 
\end{multline*}
By comparing with Proposition \ref{F_3} (ii), we obtain 
\begin{multline*}
F_3(\a,\a',\b,\b',\a+\a'+1; x,1) \\
= \frac{\vG(\a+\a'+1)\vG(\a-\b'+1)}{\vG(\a+1)\vG(\a+\a'-\b'+1)}
{}_3F_2\left({{\a,\b,\a-\b'+1} \atop {\a+1,\a+\a'-\b'+1}};x \right). 
\end{multline*}
In particular, we have
$$\wt{F}(\a,\b)
= \frac{1}{\a} \frac{\vG(\a)\vG(\b)}{\vG(\a+\b)} {}_3F_2\left({{\a,\a,1} \atop {\a+1,\a+\b}}; 1\right) . 
$$
By using Dixon's formula (cf. \cite{slater}, 2.3.3) 
$${}_3F_2\left({{a,b,c} \atop {d,e}};1\right)
= \frac{\vG(d)\vG(e)\vG(s)}{\vG(a)\vG(b+s)\vG(c+s)}
{}_3F_2\left({{d-a,f-a,s} \atop {b+s,c+s}};1\right),$$
where $s=d+e-a-b-c$, repeatedly, we obtain three other expressions.  
In particular, we have 
\begin{equation}\label{Fand3F2}
\wt{F}(\a,\b) 
=  \left(\frac{\vG(\a)\vG(\b)}{\vG(\a+\b)}\right)^2 
{}_3F_2\left({{\a,\b,\a+\b-1} \atop {\a+\b,\a+\b}};1\right), 
\end{equation}
which is symmetric and has better convergence. 

On the other hand, 
Ross \cite{ross-cr} and Kimura \cite{k-kimura} also studied the element
\begin{equation*}
\{1-xy,x\} = -\{1-xy,y\} \ \in H_\sM^2(X_N,\Q(2))_\Z
\end{equation*}
(the tame symbols vanish). 
We explain that its study is in fact equivalent to the study of $e_N^{[1,1]_k}$, or equivalently, of $e_N^{1,1}$. 
Let the curve $C_N^{1,1}$ and the morphism of degree $N$ 
$$\psi\colon X_N \lra C_N^{1,1}$$ 
be as in Remark \ref{C_N}. 
The automorphism group (over $K_N$) of $X_N/C_N^{1,1}$ is $\bigl\{g_N^{r,-r} \bigm| r \in \Z/N\Z\bigr\}$. 
One sees easily 
$$\{1-xy,x\} = \frac 1 N \psi^*\{1-v,u\}.$$
As Yasuda pointed out to the author, we can prove that 
\begin{equation}\label{yasuda}
\psi_* e_N = 3\{1-v,u\} = 3 \psi_*\{1-xy,x\}
\end{equation}
in $H_\sM^2(C_N^{1,1},\Q(2))_\Z$.
Therefore, we have
$$3\{1-xy,x\}= \frac{1}{N} \psi^*\psi_* e_N = \frac{1}{N} \left(\sum_{r\in\Z/N\Z}\nolimits g_N^{r,-r}\right) e_N.$$
Since 
$$ \sum_r g_N^{r,-r} p_N^{a,b} = \begin{cases}
Np_N^{a,b} & \text{if $a=b$},\\
0 & \text{otherwise}, 
\end{cases}
$$
we obtain 
\begin{equation}\label{1-xy,x}
3\{1-xy,x\} = \left(\sum_{a\in \Z/N\Z}\nolimits p_N^{a,a} \right)e_N.
\end{equation}
In particular, the regulator of $\{1-xy,x\}$ is calculated 
by the regulators of $e_N^{a,a}$, which then reduces to the study of $e_{N'}^{1,1}$ for some $N'|N$.  

If we apply a similar calculation as for $e_N$ to $\{1-xy,x\}$, 
we obtain similar results as Theorem \ref{main-theorem} and its corollaries for $X_N^{a,a}$ and $X_N^{[1,1]_k}$. 
Then we encounter with another generalized hypergeometric function of one variable,  
namely, 
\begin{equation*}
G(\a,\b,\g;x) := \sum_{n \geq 0} \frac{(\a,n)(\b,n)}{(\g,2n)} x^n 
\end{equation*}
and its special value 
\begin{equation*}
\wt G(\a,\b) := \frac{\vG(\a)\vG(\b)}{\vG(\a+\b+1)} G(\a,\b,\a+\b+1;1).
\end{equation*}
In fact, it is again a special case of ${}_3F_2$: 
\begin{equation}\label{Gand3F2}
G(\a,\b,\g;x) = {}_3F_2\left({{\a,\b,1} \atop {\frac{\g}{2}, \frac{\g+1}{2}}}; \frac{x}{4}\right). 
\end{equation}
We remark that it converges for $|x|<4$, 
and $x=1$ is not on the boundary.  
By \eqref{1-xy,x} and the comparison of the regulators, we obtain: 
\begin{equation*}
\wt F\bigl(\tfrac{\angle{a}}{N}, \tfrac{\angle{a}}{N}\bigr)-\wt F\bigl(\tfrac{\angle{-a}}{N}, \tfrac{\angle{-a}}{N}\bigr)
 = 3\left(\wt G\bigl(\tfrac{\angle{a}}{N}, \tfrac{\angle{a}}{N}\bigr)-\wt G\bigl(\tfrac{\angle{-a}}{N}, \tfrac{\angle{-a}}{N}\bigr)\right)
\end{equation*}
for any $a \neq 0$. It follows that 
$$\wt F(\a,\a)-\wt F(1-\a,1-\a) = 3 \left(\wt G(\a,\a)-\wt G(1-\a,1-\a)\right)$$
for any $\a \in \C$ with $0<\Re(\a)<1$, for the both sides are holomorphic with respect to $\a$. 
It seems  that 
\begin{equation}\label{FandG}
\wt F(\a,\a) = 3 \wt G(\a,\a)
\end{equation} 
for any $\a\in\C$ with $\Re(\a)>0$, but
$\wt F(\a,\b) \neq 3 \wt G(\a,\b)$ in general. 
By \eqref{Fand3F2} and \eqref{Gand3F2}, \eqref{FandG} is equivalent to: 
$$\frac{\vG(\a)^2}{\vG(2\a)} \, {}_3F_2\left({{\a,\a,2\a-1} \atop {2\a,2\a}};1\right)
= \frac{3}{2\a}\, {}_3F_2\left({{\a,\a,1} \atop {\a+\frac{1}{2},\a+1}};\frac{1}{4}\right).$$
The author does not know if such a relation is known to the experts. 

\begin{remark}
We could also study the element $\{1-x^ry^s,x\}$, though its tame symbols 
do not vanish in general. 
Then, the hypergeometric function involved should be  
$$\sum_{n \geq 0} \frac{(\a,rn)(\b,sn)}{(\g,(r+s)n)}x^n, \quad \g=\a+\b+1,$$
which is also written as  
${}_{p}F_{q}\left({{\a_1,\cdots, \a_{p}} \atop {\b_1,\cdots, \b_{q}}};\frac{x}{R}\right)$
with $p=q+1=r+s+1$, suitable $\a_i$, $\b_j$ and $R>1$. 
\end{remark}

\subsection{Action of the symmetric group}

The Fermat curve has another symmetry, namely, the action of the symmetric group.   
Using this, we construct more elements in motivic cohomology.

Let us suppose for simplicity that $N$ is odd, so that the equation \eqref{equation-fermat} of $X_N$ is written as: 
$$x_0^N+y_0^N+(-z_0)^N=0.$$
The symmetric group $S_3$ of degree $3$ acts on $X_N$ as permutations on the set 
$\{x_0,y_0,-z_0\}$. 
Since  
$$(1\ 2)^*e_N = \{1-y,1-x\} = -e_N,$$ 
we do not get a new element by $(1\ 2)$. 
The quotient of $S_3$ by the subgroup generated by $(1 \ 2)$ is represented by 
$(1)$, $(1\ 3)$ and $(2\ 3)$. 

For $(a,b) \in G_N$, put 
\begin{equation*}
c=-a-b \in \Z/N\Z,
\end{equation*} 
and use redundant notations with index $(a,b,c)$ instead of $(a,b)$, such as
$p_N^{a,b,c} = p_N^{a,b}$, $e_N^{a,b,c}=e_N^{a,b}$. 

\begin{definition}
Define elements of $H_\sM^2(X_N^{a,b},\Q(2))_\Z$ by 
$$e_{(1)}^{a,b,c}= e_N^{a,b,c}, \quad 
e_{(1\ 3)}^{a,b,c}= p_N^{a,b,c}(1\ 3)^*e_{N,K_N}, \quad 
e_{(2\ 3)}^{a,b,c}= p_N^{a,b,c}(2 \ 3)^*e_{N,K_N}, 
$$
where we put $e_{N,K_N}=\pi_{K_N/k}^*e_N$. 
\end{definition}

\begin{lemma}\label{projector-twist}  
In $\End_{\sM_{K_N,E_N}}(X_{N,K_N})$, we have
$$p_N^{a,b,c} \circ (1\ 3)^* = (1\ 3)^* \circ  p_N^{c,b,a}, \quad 
p_N^{a,b,c} \circ (2\ 3)^* = (2\ 3)^* \circ  p_N^{a,c,b}.$$ 
\end{lemma}

\begin{proof} 
Since $ (1\ 3)\circ g_N^{r,s} = g_N^{-r,-r+s}\circ  (1\ 3)$ in $\sV_k$, we have 
\begin{multline*}
N^2 p_N^{a,b,c} \circ (1\ 3)^* 
=  \sum_{r,s} \theta_N^{a,b}(g_N^{r,s})^{-1} g_N^{r,s} \circ  (1\ 3)^* \\
=  (1\ 3)^* \circ \sum_{r,s} \theta_N^{a,b}(g^{r,s})^{-1} g_N^{-r,-r+s} 
=  (1\ 3)^* \circ \sum_{r',s'} \theta_N^{a,b}(g_N^{-r',-r'+s'})^{-1} g_N^{r',s'} \\ 
=   (1\ 3)^* \circ \sum_{r',s'} \theta_N^{-a-b,b}(g^{r',s'})^{-1} g_N^{r',s'}  
= N^2  (1\ 3)^* \circ p_N^{c,b,a}.
\end{multline*}
The other one is parallel. 
\end{proof}

Put \ $\eta_N=e^{\frac{2\pi i}{N}} \in \C^*$, 
and a polynomial
\begin{equation*}
\Phi_N(T) = T^{\frac{N-1}{2}}-T^{\frac{1-N}{2}}.
\end{equation*}
Obviously, $\Phi_N(\eta_N^a) \in i\R$ and  
$\Phi_N(\eta_N^{-a})=-\Phi_N(\eta_N^a)$. 
\begin{lemma}\label{regulator-twist}
 Let $(a,b) \in I_N$ and $\wt\om_N^{a,b,c} = \wt\om_N^{a,b} \in H^1(X_N(\C),\C)$ be as defined in \eqref{omega-normalized}. 
 Then we have{\rm :} 
\begin{equation*}
(1\ 3)^* \wt\om_N^{a,b,c} = -\frac{\Phi_N(\eta_N^a)}{\Phi_N(\eta_N^c)}\, \wt\om_N^{c,b,a}, \quad
(2\ 3)^* \wt\om_N^{a,b,c} = -\frac{\Phi_N(\eta_N^b)}{\Phi_N(\eta_N^c)}\, \wt\om_N^{a,c,b}. 
\end{equation*}
\end{lemma}
\begin{proof} 
We only prove the first one.
First, assume that $\angle{a}+\angle{b} <N$, i.e. $\angle{a}+\angle{b}+\angle{c}=N$. Then we have 
\begin{align*}
(1\ 3)^*\om_N^{a,b} &= \left(\frac{1}{x}\right)^{\angle{a}} \left(-\frac{y}{x}\right)^{\angle{b}-N}\left(-\frac{dx}{x}\right) \\
& = (-1)^{\angle{b}} x^{N-\angle{a}-\angle{b}}y^{\angle{b}-N}\frac{dx}{x}
= (-1)^{\angle{b}} \om_N^{c,b}. 
\end{align*}
By  \eqref{beta-gamma} and the well-known relation
$$\vG(\a)\vG(1-\a) = \frac{\pi}{\sin \pi \a},$$ 
we have
\begin{equation*}
\frac{B\bigl(\tfrac{\angle{c}}{N}, \tfrac{\angle{b}}{N}\bigr)}{B\bigl(\tfrac{\angle{a}}{N}, \tfrac{\angle{b}}{N}\bigr)} 
= \frac{\vG\bigl(\frac{\angle{a}+\angle{b}}{N}\bigr)}{\vG\bigl(\frac{\angle{a}}{N}\bigr)\vG\bigl(\frac{\angle{b}}{N}\bigr)}
\frac{\vG\bigl(\frac{\angle{c}}{N}\bigr)\vG\bigl(\frac{\angle{b}}{N}\bigr)}
{\vG\bigl(\frac{\angle{c}+\angle{b}}{N}\bigr)}
= \frac{\vG\bigl(1-\frac{\angle{c}}{N}\bigr)\vG\bigl(\frac{\angle{c}}{N}\bigr)}{\vG\bigl(\frac{\angle{a}}{N}\bigr)\vG\bigl(1-\frac{\angle{a}}{N}\bigr)}
=\frac{\sin\frac{\angle{a}}{N}\pi}{\sin\frac{\angle{c}}{N}\pi}. 
\end{equation*}
Since
$(-1)^{\angle{a}}\sin\tfrac{\angle{a}}{N}\pi = - \Im \bigl(\eta_N^{\frac{N-1}{2}a}\bigr),$
we obtain the formula. 
The case $\angle{a}+\angle{b} >N$ is reduced to the first case using
\begin{align*}
c_\infty(1\ 3)^*\wt\om_N^{a,b} &= (1\ 3)^* c_\infty\wt\om_N^{a,b} = (1\ 3)^* \wt\om_N^{-a,-b}, 
\quad 
 c_\infty\wt\om_N^{c,b}=\wt\om_N^{-c,-b}.  
\end{align*}
\end{proof}

\begin{definition} For $a$, $b \in \Z/N\Z$, and an infinite place $v$ of $K_N$, 
define $\mathbf{r}_{N,v}^{a/b} \in E_{N,\R}^*$ by
$$
\s (\mathbf{r}_{N,v}^{a/b}) = - \Phi_N(\eta_N^{ha})/\Phi_N(\eta_N^{hb}) 
$$
for each $\s \colon E_N \hookrightarrow \C$, 
where $h \in H_N$ is such that $\t(\z_N)^h=\s(\xi_N)$ for $\t\colon K_N \hookrightarrow \C$ inducing $v$. 
It does not depend on the choice of $\t$.  
\end{definition}

\begin{proposition}\label{twist-regulator}
Let $N$ be odd and the notations be as in Theorem \ref{main-theorem}. 
Then we have 
\begin{align*}
r_{\sD,v} (e_{(1\ 3)}^{a,b,c}) 
= \mathbf{r}_{N,v}^{c/a} \mathbf{c}_{N,v}^{c,b,a}  \l_{N,v}^{a,b,c}, \quad
r_{\sD,v} (e_{(2\ 3)}^{a,b,c}) 
=\mathbf{r}_{N,v}^{c/b} \mathbf{c}_{N,v}^{a,c,b} \l_{N,v}^{a,b,c}. 
\end{align*}
\end{proposition}

\begin{proof}
We only prove the first one. 
By Lemma \ref{projector-twist} and Theorem \ref{main-theorem}, we have
\begin{multline*}
r_{\sD,v}(p_N^{a,b,c}(1\ 3)^*e_{N,K_N})=r_{\sD,v}((1\ 3)^*p_N^{c,b,a}e_{N,K_N})\\
= (1\ 3)^* r_{\sD,v}(p_N^{c,b,a}e_{N,K_N})
= \mathbf{c}_{N,v}^{c,b,a} (1\ 3)^*\l_v^{c,b,a} = \mathbf{c}_{N,v}^{c,b,a} \mathbf{r}_{N,v}^{c/a} \l_v^{a,b,c}, 
\end{multline*}
where the last equality follows from 
Proposition \ref{basis-deligne} and Lemma \ref{regulator-twist}. 

\end{proof}

\subsection{Examples}
Now we study two particular cases $N=5$ and $7$ with $k=\Q$. Then, for any $(a,b) \in I_N$,  
$$\dim_{E_N} H_\sM^2(X_N^{a,b},\Q(2))_\Z = \dim_\Q H_\sM^2(X_N^{[a,b]_\Q},\Q(2))_\Z$$ 
is conjectured to be $2$ and $3$, respectively. 

For brevity, we put for $(a,b) \in I_N$
\begin{equation*}
F_N^{a,b,c}
=\wt{F}\bigl(\tfrac{\angle{a}}{N},\tfrac{\angle{b}}{N}\bigr)-\wt F\bigl(\tfrac{\angle{-a}}{N},\tfrac{\angle{-b}}{N}\bigr) \ \in \R. 
\end{equation*} 
Note that 
\begin{equation*}\label{F^{a,b,c}}
F_N^{a,b,c}=-F_N^{-a,-b,-c}, \quad F_N^{a,b,c}=F_N^{b,a,c}.  
\end{equation*}
By Proposition \ref{decreasing}, $F_N^{a,b,c}>0$ if and only if $\angle{a}+\angle{b} <N$.  
Moreover, $F_N^{a,b,c}$ is monotonously decreasing with respect to each $\angle{a}$ and $\angle{b}$. 

\begin{theorem}\label{N=5}
Suppose that $k\subset \Q(\mu_5)$. 
Then the regulator map
\begin{align*}
r_\sD \ot_\Q \R\colon H_\sM^2(X_5,\Q(2))_\Z \ot_\Q \R \lra H_\sD^2(X_{5,\R},\R(2))
\end{align*}
is surjective.  
\end{theorem}

\begin{proof}
It suffices to prove the surjectivity for $X_5^{a,b}$ for any $(a,b) \in I_5$. 
Since the surjectivity depdnds only on the class $[a,b]_\Q$, it suffices to prove it for 
$(a,b)=(1,1)$, $(1,2)$, and $(2,1)$. 

For an embedding $\s \colon E_5 \hookrightarrow \C$, 
define  $\t_1, \t_2 \colon K_5 \hookrightarrow \C$  by 
$\t_1(\z_5)=\s(\xi_5), \t_2(\z_5)^2=\s(\xi_5)$, and let  $v_i$ be the infinite place of $K_5$ induced by $\t_i$. 
By Theorem \ref{main-theorem} and Proposition \ref{twist-regulator}, 
we have 
$$
\begin{pmatrix}
r_{\sD}(e_{(1)}^{a,b,c})_\s & r_\sD(e_{(1\ 3)}^{a,b,c})_\s
\end{pmatrix}
= -\frac{1}{4\cdot 5^2\pi i} 
\begin{pmatrix}
\l_{5,v_1}^{a,b,c} & \l_{5,v_2}^{a,b,c}
\end{pmatrix}
A^{a,b,c}
$$
with
$$
A^{a,b,c}=
\begin{pmatrix}
F_5^{a,b,c} &  -\Phi_5(\eta_5^c)\Phi_5(\eta_5^a)^{-1}F_5^{c,b,a} \\ 
F_5^{2a,2b,2c} & -\Phi_5(\eta_5^{2c}) \Phi_5(\eta_5^{2a})^{-1}F_5^{2c,2b,2a}
\end{pmatrix}. 
$$ 

First, let $(a,b,c) = (1,1,3)$. Then one calculates 
\begin{align*}
\det(A^{1,1,3}) = \frac{\eta_5^2-\eta_5^{-2}}{\eta_5-\eta_5^{-1}}F_5^{1,1,3}F_5^{1,2,2}
+ \frac{\eta_5-\eta_5^{-1}}{\eta_5^2-\eta_5^{-2}}F_5^{3,1,1}F_5^{2,2,1}.
\end{align*}
Since $F_5^{1,1,3},  F_5^{1,2,2}, F_5^{3,1,1}, F_5^{2,2,1} > 0$, 
it follows that  $\det (A^{1,1,3}) > 0$.  

Secondly, if $(a,b,c)=(1,2,2)$, then 
\begin{align*}
\det(A^{1,2,2}) &= - \frac{\eta_5^2-\eta_5^{-2}}{\eta_5-\eta_5^{-1}}F_5^{1,2,2}F_5^{4,4,2}
- \frac{\eta_5-\eta_5^{-1}}{\eta_5^2-\eta_5^{-2}}F_5^{2,2,1}F_5^{2,4,4}\\
&= \frac{\eta_5^2-\eta_5^{-2}}{\eta_5-\eta_5^{-1}}F_5^{1,2,2}F_5^{1,1,3}
+ \frac{\eta_5-\eta_5^{-1}}{\eta_5^2-\eta_5^{-2}}F_5^{2,2,1}F_5^{3,1,1}\\
&= \det(A^{1,1,3}) >0.
\end{align*}

Finally, by the symmetry, the remaining case $(a,b,c) = (2,1,2)$ is proved by using 
$e_{(2\ 3)}^{a,b,c}$ instead of $e_{(1\ 3)}^{a,b,c}$. 
\end{proof}

By a more precise argument similar to the proof of Corollary \ref{346}, we obtain:
\begin{corollary}\label{cor-N=5}
Let $(a,b) \in I_5$ and assume the Beilinson conjecture (Conjecture \ref{conj-beilinson}) for 
$X_5^{[a,b]_\Q} \in \sM_\Q$. 
Then it follows that
\begin{align*}
L(j_5^{a,b},2) \equiv \pi^2 L^*(j_5^{a,b},0)
\equiv 
\frac{\sin\frac{4\pi}{5}}{\sin\frac{2\pi}{5}}F_5^{1,1,3}F_5^{1,2,2}
+ \frac{\sin\frac{2\pi}{5}}{\sin\frac{4\pi}{5}}F_5^{3,1,1}F_5^{2,2,1}
\end{align*}
modulo $\Q^*$.  
\end{corollary}

\begin{proof}
Just note that a basis of $H_1(X^{[a,b]_\Q}_5(\C),\Q)^{-}$ is given by
\begin{align*}
&\bigl((\xi_5-\xi_5^{-1})(p_5^{a,b}-p_5^{-a,-b})+ (\xi_5^2-\xi_5^{-2})(p_5^{2a,2b}-p_5^{-2a,-2b})\bigr) \g_5, \\
&\bigl((\xi_5^2-\xi_5^{-2})(p_5^{a,b}-p_5^{-a,-b})-(\xi_5-\xi_5^{-1})(p_5^{2a,2b}-p_5^{-2a,-2b})\bigr) \g_5. 
\end{align*} 
\end{proof}

\begin{remark}
Kimura \cite{k-kimura} studies the curve $C_5^{1,1}$ (see Remark \ref{C_N}) over $\Q$, 
which is equivalent to the study of $X_5^{[1,1]} \in \sM_\Q$, or to $X_5^{1,1} \in \sM_{K_5,E_5}$. 
He computes numerically the determinant of the regulators of 
$$\a= \psi_* \{1-xy,x\}, \quad \b=\psi_*\left\{x+y,\frac{1-x}{y}\right\},$$
and showed that it is non-trivial. 
By \eqref{yasuda} and
$$(1\ 3)^*e_N = (1\ 3)^*\left\{\frac{1-x}{y}, \frac{1-y}{x}\right\} = \left\{\frac{1-x}{y}, x+y\right\},$$
where the first equality follows from $\{1-x,x\} = \{y,1-y\} =0$ and $N^2\{x,y\} = \{x^N,y^N\}=0$, 
his study corresponds to the study of our $e_{(1)}^{1,1}$ and $e_{(1\ 3)}^{1,1}$.  
\end{remark}

\begin{proposition}\label{N=7}
Let $N=7$, $(a,b) \in I_7$ and suppose that $k\subset\Q(\mu_7)$.  
Then the regulator map
\begin{align*}
r_\sD \ot_\Q \R\colon H_\sM^2(X_7^{[a,b]_\Q},\Q(2))_\Z \ot_\Q \R \lra H_\sD^2(X_{7,\R}^{[a,b]_\Q},\R(2))
\end{align*}
is surjective if $a$, $b$ and $c$ are different to each other. 
Otherwise, the dimension of\, $\Im(r_{\sD})$ is at least $2$. 
\end{proposition}

\begin{proof}
The surjectivity is equivalent to that for $X_7^{a,b}$. The regulators of $e_{(1)}^{a,b,c}$, $e_{(1\ 3)}^{a,b,c}$ and $e_{(2\ 3)}^{a,b,c}$ are expressed by the matrix
$$B^{a,b,c}=
\begin{pmatrix}
F_7^{a,b,c} & - \frac{\Phi_7(\eta_7^c)}{\Phi_7(\eta_7^a)}F_7^{c,b,a} & -\frac{\Phi_7(\eta_7^c)}{\Phi_7(\eta_7^b)}F_7^{a,c,b} \\
F_7^{2a,2b,2c} & -\frac{\Phi_7(\eta_7^{2c})}{\Phi_7(\eta_7^{2a})} F_7^{2c,2b,2a} &  -\frac{\Phi_7(\eta_7^{2c})}{\Phi_7(\eta_7^{2b})}F_7^{2a,2c,2b}\\
F_7^{3a,3b,3c} &  -\frac{\Phi_7(\eta_7^{3c})}{\Phi_7(\eta_7^{3a})}F_7^{3c,3b,3a} &  -\frac{\Phi_7(\eta_7^{3c})}{\Phi_7(\eta_7^{3b})}F_7^{3a,3c,3b}
\end{pmatrix}. 
$$

In the first case, we have $[a,b]_\Q=[1,2]_\Q$ or $[2,1]_\Q$, and by the symmetry, it suffices to treat 
$(a,b,c)=(1,2,4)$. Then one calculates: 
\begin{align*}
\det(B^{1,2,4})
&= C (s^3+t^3+u^3-3stu)\\ 
&= \frac{C}{2}
(s+t+u)\left\{(s-t)^2+(t-u)^2+(u-s)^2\right\}
\end{align*}
with
\begin{equation*}
s=\frac{iF_7^{1,2,4}}{\Phi_7(\eta_7^4)}, \quad t=\frac{iF_7^{2,4,1}}{\Phi_7(\eta_7)}, 
\quad u=\frac{iF_7^{3,6,5}}{\Phi_7(\eta_7^5)},  
\quad C= i\Phi_7(\eta_7^4)\Phi_7(\eta_7)\Phi_7(\eta_7^5) . 
\end{equation*}
Note that $F_7^{1,2,4}>F_7^{1,4,2}>F_7^{2,4,1} >0$. 
Since $s,u<0<t$, we have $(s-t)^2+(t-u)^2+(u-s)^2 \neq 0$. 
On the other hand, $s+t+u \neq 0$ since 
$$-s-u >\left(-\frac{i}{\Phi_7(\eta_7^4)} -\frac{i}{\Phi_7(\eta_7^5)}\right) F_7^{2,4,1}
= t.$$
Hence we obtained $\det(B^{1,2,4})\neq 0$. 

In the second case, we are reduced to consider $(a,b,c)=(1,1,5)$. 
Then we have $r_{\sD}(e_{(1\ 3)}^{a,b})=r_{\sD}(e_{(2\ 3)}^{a,b})$. 
However, the minor matrix of $B^{1,1,5}$ 
$$\begin{pmatrix}
F_7^{1,1,5} & -\frac{\Phi_7(\eta_7^5)}{\Phi_7(\eta_7)} F_7^{5,1,1} \\ 
F_7^{2,2,3} & -\frac{\Phi_7(\eta_7^3)}{\Phi_7(\eta_7^2)}F_7^{3,2,2}
\end{pmatrix}
$$
has non-trivial determinant since $F_7^{1,1,5}, F_7^{5,1,1}, F_7^{2,2,3}, F_7^{3,2,2} >0$, 
and 
$$\Phi_7(\eta_7^3)\Phi_7(\eta_7^2)^{-1}<0< \Phi_7(\eta_7^5)\Phi_7(\eta_7)^{-1}.$$
\end{proof}

Similarly as Corollary \ref{346} and Corollary \ref{cor-N=5},  we obtain: 
\begin{corollary}
Let $(a,b) \in I_7$ and assume that $a, b$ and $c$ are different to each other. 
If the Beilinson conjecture (Conjecture \ref{conj-beilinson}) holds for $X_7^{[a,b]_\Q}$, then it follows that
\begin{align*}
 L(j_7^{a,b},2)\equiv 
\pi^3 L^*(j_7^{a,b},0)  \equiv
 s^3+t^3+u^3-3stu 
\end{align*}
modulo $\Q^*$, 
where 
$$s= -\frac{F_7^{1,2,4}}{\sin \frac{4\pi}{7}}, \quad t= \frac{F_7^{2,4,1}}{\sin \frac{6\pi}{7}}, \quad
u= -\frac{F_7^{4,1,2}}{\sin \frac{2\pi}{7}}.$$
\end{corollary}

\begin{remark} 
As above, for any odd $N \geq 5$ and $(a,b)\in I_N$, we can always find two of 
$\bigl\{e_{(1)}^{a,b}, e_{(1\ 3)}^{a,b}, e_{(2\ 3)}^{a,b}\bigr\}$ whose regulators are linearly independent. 
\end{remark}


\end{document}